\documentclass{amsart}
\usepackage{lineno,hyperref}
\usepackage[utf8]{inputenc}
\usepackage[english]{babel}
\usepackage{color}
\usepackage{lineno,hyperref}
\usepackage{upgreek}
\usepackage{amsmath,amssymb,amsthm,amsfonts,graphicx,hyperref,mathrsfs,latexsym,filecontents,xcolor}
\newtheorem{thm}{Theorem}[section]
\newtheorem{theorem}[thm]{Theorem}
\newtheorem{corollary}[thm]{Corollary}
\newtheorem{lemma}[thm]{Lemma}

\theoremstyle{definition}

\newtheorem{rem}[thm]{Remark}
\numberwithin{equation}{section}

\begin{document}

\title[Stability Theory for Nullity and Deficiency of Linear Relations]
 {Stability Theory for Nullity and Deficiency of Linear Relations}

\author[Luliro]{Kito Silas Luliro}
\address{Department of Mathematics, Makerere University, P. O. Box 7062, Kampala, Uganda}
\email{ksilas@cedat.mak.ac.ug \& slkito2020@gmail.com}

\author[Wanjala]{Gerald Wanjala}
\address{Department of Mathematics, Sultan Qaboos University, P. O. Box 36, PC 123 Al Khoud, Sultanate of Oman}
\email{gwanjala@squ.edu.om \& wanjalag@yahoo.com}


\footnote{This work was financially supported by the Swedish Sida Phase-IV
bilateral program with Makerere University, 2015-2020, project 316
Capacity building in Mathematics and its applications".}
\begin{abstract}
Let $\mathcal A$ and $\mathcal B$ be two closed linear relation acting between two Banach spaces $X$ and $Y$ and let $\lambda$ be a complex  number. We study the stability of the nullity and deficiency of $\mathcal A$ when it is perturbed by $\lambda\mathcal B$. In particular, we show the existence of a constant $\rho>0$ for which both the nullity and deficiency of $\mathcal A$ remain stable under perturbation by $\lambda\mathcal B$ for all $\lambda$ inside the disk $\vert \lambda\vert <\rho$.
\end{abstract}
\maketitle
\section{Introduction}

For purposes of introduction we shall consider  bounded linear operators $A$ and $B$ with domain $X$ and range in $Y$. As usual, let $N(A)$ and $R(A)$ denote the null space and range of $A$ respectively. The dimensions of $N(A)$ and $Y\slash R(A)$ are called the nullity and the
deficiency of $A$ respectively and denoted by $\alpha(A)$ and $\beta(A)$. It is well known that $\alpha(A)$ and $\beta(A)$ have some kind of stability when $A$ is subjected to some kind of perturbation (see for example ~\cite{kato2013perturbation}). More precisely, $\alpha(A)$ and $\beta(A)$ are unchanged when $A$ is perturbed by some bounded linear operator $B$ under certain prescribed conditions. This stability can be described in the form
$$
\alpha(A-B)- \beta(A-B)=\alpha(A)-\beta(A).
$$
Another convenient way of describing this stability is to put it in the form
$$
\alpha(A-B)=\alpha(A) \; \; {\rm and}\; \; \beta(A-B)=\beta(A).
$$
The stability concept described here is very useful in studying eigenvalue problems of the form $Ax=\lambda Bx$ and $A^*y=\lambda B^*y$, where $A^*$ denotes the adjoint operator.

This paper deals with the stability theory for nullity and deficiency of linear relations and it can be
seen as a generalization of the classical theory for the corresponding quantities for linear operators.
The theory and exposition developed here goes along the lines of the classical texts on the perturbation theory for linear operators (see for example ~\cite{kato2013perturbation} and ~\cite{kato1958perturbation}), but in a more general setting. Some stability theorems for multivalued linear operators or what we refer to here as linear relations, have been considered in ~\cite{cross1998multivalued} and more recently in ~\cite{ren2018stability}. In either of these cases, the perturbing multivalued linear operator $\mathcal B$ does not vary with the varying $\lambda$ as the case we consider here.
\section{Preliminaries}
\subsection{Relations on sets}
In this section we introduce some notation and consider some basic concepts concerning relations on sets. Let $U$ and $V$ be two nonempty sets. By  a \emph{relation} $\mathcal T$ from $U$ to $V$ we mean a mapping whose domain $D(\mathcal T)$ is a nonempty subset of $U$, and taking values in $2^{V}\setminus \emptyset$, the collection of all nonempty subsets of $V$. Such a mapping
$\mathcal T$ is also referred to as a \emph{multi-valued} operator or at times as a \emph{set valued function}. If $\mathcal T$ maps the elements of its domain to singletons, then $\mathcal T$ is said to be a \emph{single valued} mapping or operator. Let $\mathcal T$ be a relation from $U$ to $V$ and let $\mathcal T(u)$ denote the image of an element $u\in U$ under $\mathcal T$. If we define $\mathcal T(u) = \emptyset$ for $u\in U$ and $u\notin D(T)$ then the domain $D(\mathcal T)$ of $\mathcal T$ is given by
$$
D(\mathcal T) = \{ u\in U : \mathcal T(u) \ne \emptyset \}.
$$
Denote by $R(U, V)$ the class of all relations from $U$ to $V$.
If $\mathcal T$ belongs to $R(U, V)$, the graph of $\mathcal T$, which we denote by $G(\mathcal T)$ is the subset of $U \times V$ defined by
$$
G(\mathcal T)= \{ (u,v) \in U \times V : u\in D(\mathcal T), v\in \mathcal T(u)\}.
$$
A relation $\mathcal T\in R(U, V)$ is uniquely determined by its graph, and conversely any nonempty subset of $U\times V$ uniquely determines a relation $\mathcal T\in R(U, V)$.

For a relation $\mathcal T\in R(U, V)$ we define its inverse $\mathcal T^{-1}$ as the relation from $V$ to $U$ whose graph $G(\mathcal T^{-1})$ is given by
\begin{equation}\label{Inverse}
G(\mathcal T^{-1})= \{ (v, u) \in V \times U : (u, v) \in G(T)\}.
\end{equation}
Let $\mathcal T\in R(U, V)$. Given a subset $M$ of $U$, we define the image of $M$, $T(M)$ to be
$$
\mathcal T(M) = \bigcup \{ T(m) : m\in  M \cap D(T)\}.
$$
With this notation we define the range of $\mathcal T$ by
$$
R(\mathcal T): = \mathcal T(U)
$$

Let $N$ be a nonempty subset of $V$. The definition of $\mathcal T^{-1}$ given in \eqref{Inverse} above implies that
\begin{equation}\label{Inverse-set}
\mathcal T^{-1}(N)=\{ u\in D(\mathcal T):N\cap \mathcal T(u) \ne \emptyset\}.
\end{equation}
If in particular $v\in R(\mathcal T)$, then
$$
\mathcal T^{-1}(v) = \{ u\in D(\mathcal T): v\in \mathcal T(u)\}.
$$

For a detailed study of relations, we refer to ~\cite{arens1961operational}, ~\cite{coddington1978positive}, ~\cite{coddington1973extension}, ~\cite{azizov2013compressions}, ~\cite{cross1998multivalued}, and ~\cite{wanjala2015operator}.

\subsection{Linear Relations}
Let $X$ and $Y$ be Linear spaces over a field $\mathbb K = \mathbb R$ (or $\mathbb C$) and let $\mathcal T\in R(X,Y)$. We say that $\mathcal T$ is a \emph{linear relation} or a \emph{multi-valued linear operator} if for all $x, z \in D(\mathcal T)$ and any nonzero scalar $\upalpha$ we have
\begin{enumerate}
\item[(1)] \quad $T(x) + \mathcal T(z) = \mathcal T(x+z)$,
\item[(2)] \quad $\upalpha \mathcal T(x) = \mathcal T(\upalpha x)$.
\end{enumerate}
The equalities in (1) and (2) above are understood to be set equalities. These two conditions indirectly imply that the domain of a linear relation is a linear subspace. The class of linear relations in $R(X,Y)$ will be denoted by $LR(X,Y)$. If $X=Y$ then we denote $LR(X,X)$ by $LR(X)$. We say that $\mathcal T$ is a linear relation in $X$ if $\mathcal T\in LR(X)$. We shall use the term \emph{operator} to refer to a single valued linear operator
while a multi-value linear operator will be generally referred to as a \emph{linear relation}.

If $X$ and $Y$ are normed linear spaces, we say that $\mathcal T\in LR(X,Y)$ is closed if its graph $G(\mathcal T)$ is a closed subspace of $X\times Y$. The collection of all such $\mathcal T$ will be denoted by $CLR(X,Y)$.

We conclude this section with the following theorems which are taken from ~\cite{cross1998multivalued}.
\begin{theorem}\label{Theorem 4}
Let $\mathcal T\in R(X,Y)$. The following properties are equivalent.
\begin{itemize}
\item[(i)] $\mathcal T$ is a linear relation.
\item[(ii)] $G(\mathcal T)$ is a linear subspace of $X\times Y$.
\item[(iii)] $\mathcal T^{-1}$ is a linear relation.
\item[(iv)] $G(\mathcal T^{-1})$ is a linear subspace of $Y\times X$.
\end{itemize}
\end{theorem}

\begin{corollary}\label{corollary}
Let $\mathcal T \in R(X,Y)$.
\begin{itemize}
  \item[(i)]  Then $\mathcal T$ is a linear relation  if and only if
  $$
  \mathcal T(\upalpha x_1 + \upbeta x_2) = \upalpha \mathcal T(x_1) + \upbeta \mathcal T(x_2)
  $$
  holds for all $x_1, x_2 \in D(\mathcal T)$ and some nonzero scalars $\upalpha$ and $\upbeta$.
  \item[(ii)] If $\mathcal T$ is a linear relation then $\mathcal T(0)$ and $\mathcal T^{-1}(0)$ are linear subspaces.
\end{itemize}
\end{corollary}
For a linear relation $\mathcal T$, the subspace $\mathcal T^{-1}(0)$ is called the \emph{null space} (or \emph{kernel}) of $\mathcal T$
and is denoted by $N(\mathcal T)$.
\begin{theorem}\label{Image}
Let $\mathcal T$ be a linear relation in a linear space $X$ and let $x\in D(\mathcal T)$. Then $y\in \mathcal T(x)$ if and only if
$$
\mathcal T(x)= \mathcal T(0) + y.
$$
\end{theorem}
Theorem \ref{Image} shows that $\mathcal T$ is single valued if and only if $T(0)= \{0\}$.

\begin{theorem}
Let $\mathcal T\in R(X,Y)$. Then $\mathcal T$ is a linear relation if and only if for all $x_1, x_2 \in D(\mathcal T)$ and
all scalars $\upalpha$ and $\upbeta$,
$$
\upalpha \mathcal T(x_1) + \upbeta\mathcal T(x_2) \subset \mathcal T(\upalpha x_1 + \upbeta x_2).
$$
\end{theorem}
\begin{theorem}
Let $\mathcal T\in LT(X,Y)$. Then
\begin{itemize}
\item[(a)] $\mathcal T(M+N)=\mathcal TM=\mathcal TN$ for $M\subset X$ and $N\subset D(\mathcal T)$.
\item[(b)] $\mathcal T\mathcal T^{-1}(M)=M\cap R(\mathcal T)+\mathcal T(0)$ for $M\subset Y$.
\item[(c)] $\mathcal T^{-1}\mathcal T(M)=M\cap D(\mathcal T)+\mathcal T^{-1}(0)$ for $M\subset X$.
\end{itemize}
\end{theorem}

\subsection{Normed linear relations}
Let $X$ be a normed linear space. By $B_X$ we shall mean the set
$$
B_X:=\{x\in X:\vert x\vert \leq 1\}.
$$
For a closed linear subspace $E$ of $X$, we denote by $Q_E$ the natural quotient map with domain $X$ and null space $E$. For $\mathcal T\in LR(X,Y)$, we shall denote $Q_{\overline{\mathcal T(0)}}$ by $Q_{\mathcal T}$. It is well known that for $\mathcal T\in LR(X,Y)$, the operator $Q_{\mathcal T}\mathcal T$ is single valued (see ~\cite{cross1998multivalued}).

For $\mathcal T\in LR(X,Y)$, we set $\Vert \mathcal Tx\Vert = \Vert Q_{\mathcal T}\mathcal Tx\Vert$ for $x\in D(\mathcal T$ and $\Vert \mathcal T\Vert =\Vert Q_{\mathcal T}\mathcal T\Vert$. Note that these notions do not define a norm since nonzero relations can have zero norm.
\begin{lemma}\label{T-0}
Let $\mathcal A, \mathcal B \in CLR(X,Y)$ be such that $D(\mathcal B)\supset D(\mathcal A)$ and $\mathcal B(0)\subset \mathcal A(0)$. If $x_1, x_2\in D(\mathcal A)$ are such that $\mathcal A(x_1)\cap \mathcal B(x_2)\ne \emptyset$ then $\mathcal A(x_1)-\mathcal B(x_2)\subset \mathcal A(0)$.
\begin{proof}
Let $z\in \mathcal A(x_1)\cap \mathcal B(x_2)$. Since $Q_{\mathcal A}$ and $Q_{\mathcal B}$ are single valued, we see that
$$
Q_{\mathcal A}(\mathcal A(x_1)-\mathcal B(x_2))=Q_{\mathcal A}\mathcal A(x_1)-Q_{\mathcal A}\mathcal B(x_2)=\widetilde{z}-\widetilde{z}=\widehat{0}.
$$
Hence $\mathcal A(x_1)-\mathcal B(x_2)\in \mathcal A(0)$.
\end{proof}
\end{lemma}

The following lemma is proved in ~\cite{cross1998multivalued}.

\begin{lemma}\label{Equivalent-Closed}
The following properties are equivalent for a linear relation $\mathcal A$.
\begin{itemize}
\item[(i)] $\mathcal A$ is closed.
\item[(ii)] $Q_{\mathcal A}\mathcal A$ is closed and $\mathcal A(0)$ is closed.
\end{itemize}
\end{lemma}
\begin{lemma}\label{Norm Difference}

\begin{itemize}
\item[(a)] Let $\mathcal T\in LR(X,Y)$ be bounded. Then $\Vert \mathcal T x\Vert \leq \Vert \mathcal T\Vert \Vert x\Vert$.
\item[(b)] For $\mathcal S, \; \mathcal T \in LR(X,Y)$ with $D(\mathcal S)\subset D(\mathcal T)$ and $\mathcal T(0)\subset \mathcal S(0)$
we have
$$
\Vert \mathcal Sx+\mathcal Tx\Vert \geq \Vert \mathcal Sx\Vert -\Vert \mathcal Tx\Vert.
$$
\end{itemize}
\end{lemma}
\begin{proof}
\begin{itemize}
\item[(a)] From ~\cite[II.1.6]{cross1998multivalued} we have
$\Vert \mathcal T \Vert =\underset{x\in B_{D(\mathcal T)}}{{\rm sup}}\Vert \mathcal Tx\Vert$ so that
\begin{eqnarray*}
  \Vert \mathcal T\Vert &=& \underset{x\in D(\mathcal T)}{{\rm sup}}\left\Vert \frac{1}{\Vert x\Vert}\mathcal Tx\right\Vert \; {\rm and} \\
  \Vert \mathcal T\Vert &\geq & \left\Vert \frac{1}{\Vert x\Vert}\mathcal Tx\right\Vert, \; x\in D(\mathcal T).
\end{eqnarray*}
The inequality  $\Vert \mathcal T\Vert \Vert x\Vert \geq \Vert \mathcal Tx\Vert\; {\rm for\; all}\; x\in D(\mathcal T)$ then follows from
~\cite[II.1.5]{cross1998multivalued}.
\item[(b)] Since $\mathcal T(0) \subset \mathcal S(0)$, we see that $(\mathcal S+\mathcal T)(0)=\mathcal S(0)+\mathcal T(0)=\mathcal S(0)$
since $\mathcal S(0)$ is a subspace (linear subset). For $x\in D(\mathcal S)$, let $s \in \mathcal S(x)$ and let $t\in \mathcal T(x)$. Then $s+t\in (\mathcal S+\mathcal T)(x)=\mathcal S(x) +\mathcal T(x)$ and so by ~\cite[II.1.4]{cross1998multivalued} we get
\begin{eqnarray*}
\Vert \mathcal Sx+\mathcal Tx\Vert &=& {\rm dist}\; (s+t, (\mathcal S+\mathcal T)(0)\\
&=& {\rm dist}\; (s+t, \mathcal S(0))\\
&\geq & {\rm dist}\; (s, \mathcal S(0))-{\rm dist}\; (t, (\mathcal S(0))\\
&\geq & {\rm dist}\; (s, \mathcal S(0))-{\rm dist}\; (t, (\mathcal T(0))\\
&=&\Vert \mathcal Sx\Vert -\Vert \mathcal Tx\Vert.
\end{eqnarray*}
\end{itemize}
\end{proof}

Let $X$ be a normed space. By $X^\prime$ we denote the norm dual of $X$, that is, the space of all continuous linear functionals $x^\prime$ defined on $X$, with norm
$$
\Vert x^\prime\Vert = \inf \{ \lambda : \vert [x,x^\prime]\vert \leq \lambda \Vert x\Vert \; {\rm for\; all}\; x\in X\}
$$
where $[x,x^\prime]:=x^\prime(x)$ denotes the action of $x^\prime\in X^\prime$ on $x\in X$. If $M\subset X$ and $N\subset X^\prime$, we write $M^\perp$ and $N^\top$ to mean
\begin{eqnarray*}
&& M^\perp:=\{ x^\prime\in X^\prime : [x,x^\prime]=0 \; {\rm for\; all}\;x\in M\} \;{\rm and}\\
&& N^\top:=\{ x\in X : [x,x^\prime]=0 \; {\rm for\; all}\;x^\prime\in N\}.
\end{eqnarray*}

Let $\mathcal T$ be a linear relation with $D(\mathcal T)\subset X$ and $R(\mathcal T)\subset Y$. We define the adjoint $\mathcal T^\prime $ of $\mathcal T$ by
$$
G\left(\mathcal T^\prime\right):=G\left(-\mathcal T^{-1}\right)^\perp \subset Y^\prime \times X^\prime
$$
where
$$
[(y,x),(y^\prime, x^\prime)]=[x,x^\prime]+[y,y^\prime].
$$
This means that
\begin{equation}\label{Adjoint}
(y^\prime, x^\prime)\in G\left( \mathcal T^\prime\right)\; {\rm if\; and\; only\; if}\; [y,y^\prime]-[x,x^\prime]=0\; {\rm for\; all}\; (x,y)\in G(\mathcal T).
\end{equation}
From \eqref{Adjoint} we see that $y^\prime(y)=x^\prime(x)$ for all $y\in \mathcal T(x)$, $x\in D(\mathcal T)$. Hence
\begin{equation}\label{Adjoint-2}
x^\prime \in \mathcal T^\prime (y^\prime) \; {\rm if\; and\; only\; if}\; y^\prime \mathcal T(x)=x^\prime(x)\; {\rm for\; all}\; x\in D(\mathcal T).
\end{equation}
This means that $x^\prime$ is an extension of $y^\prime\mathcal T(x)$ and therefore the adjoint $\mathcal T^\prime$ can be characterized as follows:
$$
G(\mathcal T^\prime)=\{ (y^\prime, x^\prime)\in Y^\prime \times X^\prime\; {\rm such\; that}\; x^\prime\; {\rm\; is\; an\; extension\; of}\; y^\prime\mathcal T.
$$

Please note that $\mathcal T^\prime\in CLR(Y^\prime, X^\prime)$ (see ~\cite[III.1.2]{cross1998multivalued}).

\begin{lemma}\label{Null-Space-T-0}{\rm ~\cite[III.1.4]{cross1998multivalued}}\\
Let $\mathcal T$ be a closed linear relation. Then
\begin{itemize}
\item[(a)] $N(\mathcal T^\prime)=R(\mathcal T)^\perp$.
\item[(b)] $\mathcal T^\prime (0)= D(\mathcal T)^\perp$.
\item[(c)] $N(\mathcal T)=R(\mathcal T^\prime )^\top$.
\item[(d)] $\mathcal T(0)=D(\mathcal T^\prime)^\top$.
\end{itemize}
\end{lemma}
\begin{rem}\label{Contain-Prime}
If $\mathcal T$ and $\mathcal S$ are closed linear relations with $D(\mathcal T)\subset D(\mathcal S)$ and $\mathcal S(0)\subset \mathcal T(0)$ then $\mathcal S^\prime(0)\subset \mathcal T^\prime (0)$ by Lemma \ref{Null-Space-T-0} $(b)$.
\end{rem}

\section{Lower bound of a closed linear relation}
Consider a closed linear relation $\mathcal A$ on a Banach space $X$ and let $N(\mathcal A)$ denote the null space of $\mathcal A$ which is closed since $\mathcal A$ is closed. Since $N(\mathcal A)\subset D(\mathcal A)$, a coset $\widetilde{x}\in \widetilde{X}=X/N(\mathcal A)$ which contains a point of $x\in D(\mathcal A)$ consists entirely of points of $D(\mathcal A)$. To see that this is the case, let $\widetilde{x}\in \widetilde{X}$ and let $x,y\in \widetilde{x}$ with $x\in D(\mathcal A)$. Then $y-x\in N(\mathcal A)\subset D(\mathcal A)$ and the linearity of $D(\mathcal A)$ implies that $y=x+(y-x)\in D(\mathcal A)$.  Let $\widetilde{D}$ denote the collection of all such cosets $\widetilde{x}$.
On setting
\begin{equation}\label{cofactor}
A(\widetilde{x}):=Q_{\mathcal A}\mathcal A(x) \quad {\rm for}\; \widetilde{x}\in \widetilde{D},
\end{equation}
we define a linear operator $A:\widetilde{X}\to \widehat{X}$ where $\widehat{X}:=X\slash \mathcal A(0)$. To see that  \eqref{cofactor} is well defined, let $x,y\in \widetilde{x}$. Then $x-y\in N(\mathcal A)$ and therefore
\begin{equation}\label{welldefined}
0\in \mathcal A(0)=\mathcal A(x-y)=\mathcal A(x)-\mathcal A(y).
\end{equation}
We see from \eqref{welldefined} that $\mathcal A(x)\cap \mathcal A(y)\ne \emptyset$. So, let $u\in \mathcal A(x)\cap \mathcal A(y)$. Then
$$
\mathcal A(x)=\mathcal A(0)+u=\mathcal A(y)
$$
so that $Q_{\mathcal A}\mathcal Ax=Q_{\mathcal A}\mathcal Ay$.
We have
\begin{equation}\label{equal}
D(A)=\widetilde{D}, \quad R(A)= R(Q_{\mathcal A}\mathcal A), \quad N(A)=\{ \widetilde{0}\}.
\end{equation}
\begin{rem}\label{Remark-Norm}
Since $\mathcal A(0)\subset R(\mathcal A)$, we also have that a coset $\widehat{x}\in \widehat{X}$ that contains a point of $R(\mathcal A)$ consists entirely of element of $R(\mathcal A)$. To see that this is the case, let $\widehat{x}$ be a coset in $\widehat{X}$ and let $u,v\in \widehat x$ with $u\in R(\mathcal A)$. Then $v-u\in \mathcal A(0)\subset R(\mathcal A)$. The linearity of $R(\mathcal A)$ implies that $v=u+(v-u)\in R(\mathcal A)$.
\end{rem}
\begin{lemma}\label{Lemma 1}
The linear operator $A$ defined by \eqref{cofactor} is closed . 
\end{lemma}
\begin{proof}
Let $\{ \widetilde x_n\}$ be a sequence in $\widetilde{D}$ such that $\widetilde x_n\to \widetilde x\in \widetilde{X}$ and let
$\left\{Q_{\mathcal A}\mathcal Ax_n\right\}$ be a sequence in $R(A)$ such that $Q_{\mathcal A}\mathcal Ax_n\to \widehat y\in \widehat{X}$. Let $x_n\in \widetilde{x}_n$ and $x\in \widetilde{x}$. Since  $\widetilde x_n\to \widetilde{x}$, we see that ${\rm dist} (x_n-x, N(\mathcal A))\to 0$. This means that $x_n-x$ converges to some element of $N(\mathcal A)$, say,
\begin{equation}\label{Try}
x_n-x\to u\in N(\mathcal A).
\end{equation}
From \eqref{Try} we see that $x_n\to x+u=w\in \widetilde{x}$.

Since $Q_{\mathcal A}\mathcal Ax_n\to \widehat y\in \widehat{X}$, that is, $\widehat{z}_n\to \widehat{y}$, we see that dist$(z_n-y, \mathcal A(0))\to 0$ as $n\to \infty$ and so $z_n\to y+v=z\in \widehat{y}$ for some $v\in \mathcal A(0)$ (where $z_n\in \mathcal A(x_n)$ for each $n\in \mathbb N$). The closedness of $\mathcal A$ implies that $w\in D(\mathcal A)$ and $z\in \mathcal A(w)$.
Hence $\widetilde x \in \widetilde D$ and $A(\widetilde x)=Q_{\mathcal A}\mathcal A(x) = \widehat y$, showing that $A$ is closed.
\end{proof}

We see that $A^{-1}$ is single valued since $A^{-1}\left(\widehat 0\right)=\left\{ \widetilde 0\right\}$.
We now introduce the quantity $\gamma(\mathcal A)$ called the lower bound of the linear relation $\mathcal A$. By definition,
\begin{equation}\label{lower bound}
\gamma(\mathcal A)= \frac{1}{\left\Vert A^{-1}\right\Vert}
\end{equation}
with the understanding that $\gamma(\mathcal A)=0$ if $A^{-1}$ is unbounded and that $\gamma(\mathcal A)=\infty$ if $A^{-1}=0$. It follows from \eqref{lower bound} that
\begin{equation}\label{lower bound 2}
\gamma(\mathcal A)=\sup\left\{ \gamma \in \mathbb R: \Vert \mathcal A(x)\Vert \geq \gamma \Vert \widetilde{x}\Vert={\rm dist} (x, N(\mathcal A))\; \forall x\in D(\mathcal A)\right\}.
\end{equation}
Note that $\gamma (\mathcal A)=\infty$ if and only if $\mathcal Ax=\mathcal A(0)$ for all $x\in D(\mathcal A)$. In order for  \eqref{lower bound 2} to hold even for this case, one should stipulate that $\infty \times 0=0$. Obviously $\gamma(\mathcal A)=\gamma( A)$.

Please note that characterization \eqref{lower bound 2} implies that if $\gamma(\mathcal A)=0$ then  the domain of $\mathcal A$ cannot consist of the zero element alone.

The fact that $\gamma (\mathcal A)=\infty$ if and only if $\mathcal A(x)=\mathcal A(0)$ for all $x\in D(\mathcal A)$ leads to the following lemma (see also ~\cite[Proposition II.2.2]{cross1998multivalued}).
\begin{lemma}\label{Infimum}
For $\mathcal A\in CLR(X,Y)$ we have
$$
\gamma(\mathcal A)=\left\{
\begin{array}{l}
\infty \; \; {\rm if}\; D(\mathcal A)\subset N(\mathcal A)\; {\rm and}\\
\inf \left\{ \frac{\Vert \mathcal A(x)\Vert}{\Vert \widetilde{x}\Vert}:x\in D(\mathcal A)\; \& \; x\notin N(\mathcal A)\right.\; \; {\rm otherwise}.
\end{array}
\right.
$$
\end{lemma}
\begin{rem}\label{remark}
A bounded linear operator $T$ is closed if and only if $D(T)$ is closed. 
\end{rem}
\begin{proof}
Suppose that $u_n\to u$ with $u_n\in D(T)$. The boundedness of $T$ implies that $T(u_n)$ is a Cauchy sequence and therefore converges,
say $T(u_n)\to v$. The closedness of $T$ implies that $u\in D(T)$ and $T(u)=v$. This shows that $D(T)$ is closed.
\end{proof}
If $\mathcal S$ is a closed linear relation from $X$ to $Y$, the graph of $\mathcal S$, $G(\mathcal S)$ is a closed subset of $X\times Y$. Sometimes it is convenient to regard it as a subset of $Y\times X$. More precisely, let $G^\prime(\mathcal S)$ be the linear subset of $Y\times X$ consisting of all pairs of the form $(v,u)$ where $u\in D(\mathcal S)$ and $v\in \mathcal S(u)$. We shall call $G^\prime(\mathcal S)$ the inverse graph of $\mathcal S$. As in the case of the graph $G(\mathcal S)$, $G^\prime(\mathcal S)$ is closed if and only $\mathcal S^{-1}$ is closed.
Clearly, $G(\mathcal S)=G^\prime\left(\mathcal S^{-1}\right)$. Thus $\mathcal S^{-1}$ is closed if and only $\mathcal S$ is closed.
\begin{lemma}\label{Lamma-Closed Range}
If $\mathcal A$ is a closed linear relation then $R(\mathcal A)$ is closed if and only if $\gamma(\mathcal A)>0$.
\end{lemma}
\begin{proof}
By definition $\gamma(\mathcal A)>0$ if and only if $A^{-1}$ is bounded (where $A$ is the operator defined in \eqref{cofactor}), and this is true if and only if $D\left(
A^{-1} \right)= R\left( A\right)=R(Q_{\mathcal A}\mathcal A)$ is closed (we use the fact that $ A^{-1}$ is closed because $A$ is closed, and then apply Remark \eqref{remark}).

Now assume that $\gamma(\mathcal A)>0$ and set $\left\{ y_n\right\}$ be a convergent sequence in $R(\mathcal A)$ with
\begin{equation}\label{Closed}
y_n\to y.
\end{equation}
Since $Q_{\mathcal A}$ is a bounded linear operator, the sequence $\left\{ Q_{\mathcal A}y_n\right\}$ is a Cauchy sequence in $R(Q_{\mathcal A}\mathcal A)$ and therefore converges to a point $\widehat{z}\in R(Q_{\mathcal A}\mathcal A)\subset \widehat{\mathcal H}$ since $R(Q_{\mathcal A}\mathcal A)$ is closed. We see that
dist$(y_n-z, \mathcal A(0))\to 0$ as $n\to \infty$ so that $y_n-z\to v$ for some $v\in \mathcal A(0)$, that is,
\begin{equation}\label{Closed 2}
y_n\to z+v\in \widehat{z}.
\end{equation}
Since $\mathcal A(0)\subset R(\mathcal A)$, a coset $\widehat{x}\in \widehat{H}$ that contains a point of $R(\mathcal A)$ consists entirely of element of $R(\mathcal A)$. To see that this is the case, let $\widehat{x}$ be a coset in $\widehat{\mathcal H}$ and let $u,v\in \widehat x$ with $u\in R(\mathcal A)$. Then $v-u\in \mathcal A(0)\subset R(\mathcal A)$. The linearity of $R(\mathcal A)$ implies that $v=u+(v-u)\in R(\mathcal A)$.

We see from \eqref{Closed} and \eqref{Closed 2} that $y\in \widehat z$ and that $y\in R(\mathcal A)$ since $z\in R(\mathcal A)$ and $y\in \widehat z$. This shows that $R(\mathcal A)$ is closed.

On the other hand, assume that $R(\mathcal A)$ is closed. Since $A^{-1}$ is closed (since $A$ is closed), it is enough, by the closed graph theorem, to show that $D\left( A^{-1}\right)=R(A)=R(Q_{\mathcal A}\mathcal A)$ is closed. So, assume that $\left\{ \widehat z_n\right\}$ is a sequence in $R(Q_{\mathcal A}\mathcal A)$ such that $\widehat z_n\to \widehat z\in \widehat H$. Then dist$(z_n-z, \mathcal A(0))\to 0$ as $n\to \infty$. Hence, there exists an element $w\in \mathcal A(0)$ such that $z_n\to z+w\in \widehat z$. The closedness of $R(\mathcal A)$ implies that $z+w\in R(\mathcal A)$ so that $\widehat z\in R(Q_{\mathcal A}\mathcal A)$.
\end{proof}
Please see ~\cite[III.5.3]{cross1998multivalued} for another proof of Lemma \ref{Lamma-Closed Range}.

For the definition of continuity and openness of a linear relation $\mathcal T$ mentioned in the following two lammas, please refer to ~\cite{cross1998multivalued}.

\begin{lemma}\label{Norms} {\rm ~\cite[II.3.2, III.1.3, III.1.5, III.4.6]{cross1998multivalued}}\\
Let $\mathcal S, \mathcal T\in LR(X,Y)$. Then
\begin{itemize}
\item[(a)] $\mathcal T$ is continuous if and only if $\Vert T\Vert <\infty$.
\item[(b)] $(\lambda \mathcal T)^\prime = \lambda\mathcal T^\prime \; ({\rm for}\; \lambda\neq 0)$.
\item[(c)] $\mathcal T$ is open if and only if $\gamma(\mathcal T)>0$.
\item[(d)] If $D(\mathcal S)\supset D(\mathcal T)$ and $\Vert \mathcal S\Vert < \infty$ then
$\left( \mathcal T + \mathcal S\right)^\prime = \mathcal T^\prime +\mathcal S^\prime$.
\end{itemize}
\end{lemma}

\begin{lemma}\label{Domains} {\rm ~\cite[III.4.6]{cross1998multivalued}}
\begin{itemize}
\item[(a)] $\mathcal T$ is continuous if and only if $D(\mathcal T^\prime)=\mathcal T(0)^\perp$.
\item[(b)] $\mathcal T$ is open if and only if $R(\mathcal T^\prime)= N(\mathcal T)^\perp$.
\item[(c)] If $\mathcal T$ is continuous, then $\left\Vert \mathcal T^\prime \right\Vert = \Vert \mathcal T\Vert$.
\item[(d)] If $\mathcal T$ is open, then $\gamma(\mathcal T)=\gamma(\mathcal T^\prime)$.
\end{itemize}
\end{lemma}
\section{The gap between closed linear manifolds and their dimensions}
Let $Z$ be a Banach space and let $L$ be a closed subspaces of $Z$. We denote by $S_L$ the unit sphere of $L$, that is, $S_L:=\{ u\in L:\Vert u\Vert=1\}$. For any two closed linear manifolds $M$ and $N$ of $Z$ with $M\neq \{0\}$,  define the gap between $M$ and $N$, denoted by $\delta (M,N)$ to be
$$
\delta (M,N):=\underset{u\in S_M}{\sup}\;{\rm dist}\;(u,N)
$$
and set $\delta(M, N)=0$ if $M=\{0\}$. $\delta (M,N)$ can also be characterized as the smallest number $\delta$ for which
\begin{equation}\label{Characterize}
{\rm dist}\;(u,N)\leq \delta\Vert u\Vert \; {\rm for\; all}\; u\in M.
\end{equation}
It can be seen from the definition that $0\leq \delta (M,N)\leq 1$.

See ~\cite{kato2013perturbation} for the following lemma.

\begin{lemma}
Let $M$ and $N$ be linear manifolds in a Banach space $Z$. If ${\rm dim}\;M>{\rm dim}\;N$ then there exists an $x\in M$ such that
$$
{\rm dist} (x, N)=\Vert x\Vert >0.
$$
\end{lemma}
The above lemma can be expressed in the language of the quotient space as follows.
\begin{lemma}
Let $M$ and $N$ be linear manifolds in a Banach space $Z$. If ${\rm dim}\;M>{\rm dim}\;N$ then there exists an $x\in M$ such that
$$
\Vert \widetilde{x}\Vert=\Vert x\Vert >0, \quad {\rm where}\; \widetilde{x}\in \in \widetilde{X}:=X\slash N. \; (N\; {\rm is\; closed\; since\;
dim}\;N<\infty)
$$
\end{lemma}
The following lemma is a direct consequence of the preceding one.
\begin{lemma}\label{Dimension}
If $\Vert \widetilde{x}\Vert <\Vert x\Vert$ for every none zero $x\in M$ where $\widetilde{x}\in \widetilde{X}=X\slash N$ then
${\rm dim}\; M\leq {\rm dim}\; N$.
\end{lemma}

See ~\cite[Page 200]{kato2013perturbation} and ~\cite{kato1958perturbation} for Lemma \ref{Kato page 200} and Lemma \ref{Extra} respectively.
\begin{lemma}\label{Kato page 200}
Let $M$ and $N$ be closed linear manifolds of a Banach space $Z$. If $\delta (M,N)<1$ then ${\rm dim}\; M\leq {\rm dim}\; N$.
\end{lemma}

\begin{lemma}\label{Extra}
Let $x$ be an element of a normed linear space $X$ and let $M$ and $N$ be closed linear subspaces of $X$. Consider the quotient space $\widetilde{X}:=X\slash N$ and let $\widetilde{x}$ denote the quotient class of $x$. For any $\varepsilon>0$ there exists $x_0\in \widetilde{x}$ such that
\begin{equation}\label{Extra-2}
{\rm dist}\;(x_0,M)\geq (1-\varepsilon)\left( \frac{1-\delta(M,N)}{1+\delta(M,N)}\right)\Vert x_0\Vert.
\end{equation}
\end{lemma}

\section{The quantity $\nu(\mathcal A:\mathcal B)$}
Let $X$ and $Y$ be two linear spaces and let $\mathcal A, \mathcal B \in LR(X,Y)$ with $\mathcal B(0)\subset \mathcal A(0)$. For $n\in \mathbb N$, let $M_n$ and $N_n$ be the linear manifolds of $X$ and $M_n^\prime$ and $N_n^\prime$ be the linear manifolds of $Y^\prime$ defined inductively as follows:
\begin{equation}\label{Inductively-1}
M_0=X, M_n=\mathcal B^{-1}(\mathcal A(M_{n-1}))\; \; {\rm for}\; n=1, 2, \cdots,
\end{equation}
\begin{equation}\label{Inductively-2}
N_1=\mathcal A^{-1}(0), N_n=\mathcal A^{-1}(\mathcal B(N_{n-1}))\; \; {\rm for}\; n=2, 3, \cdots,
\end{equation}
\begin{equation}\label{Inductively-3}
M_0^\prime=Y^\prime, M_n^\prime=\mathcal {B^\prime}^{-1}(\mathcal A^\prime(M_{n-1}^\prime))\; \; {\rm for}\; n=1, 2, \cdots,
\end{equation}
\begin{equation}\label{Inductively-4}
N_1^\prime=\mathcal {A^\prime}^{-1}(0), N_n^\prime=\mathcal {A^\prime}^{-1}(\mathcal B^\prime(N_{n-1}^\prime))\; \; {\rm for}\; n=2, 3, \cdots.
\end{equation}
If $M_k\supset M_{k+1}$ then $\mathcal A(M_k)\supset \mathcal A(M_{k+1})$ and therefore
\begin{equation}\label{Induction-1}
M_{k+1}=\mathcal B^{-1}(\mathcal A(M_k))\supset \mathcal B^{-1}(\mathcal A(M_{k+1}))=M_{k+2}.
\end{equation}
Since $M_0=X\supset D(\mathcal B)\supset M_1$, we conclude by induction that
\begin{equation}\label{Induction-2}
M_0\supset M_1 \supset M_2 \supset \cdots \supset N(\mathcal B).
\end{equation}
Similarly,
\begin{equation}\label{Induction-3}
N_1\subset N_2 \subset N_3 \subset \cdots \subset D(\mathcal A).
\end{equation}
Note that
\begin{equation}\label{Null-Space}
N_1= N(\mathcal{A}).
\end{equation}

\begin{lemma}\label{Equivalent}
Let $n$ be a positive integer. The following first $n$ conditions are equivalent to one another and they in turn imply that condition $(\kappa)$ holds.
\begin{itemize}
\item[(1)] $N_1\subset M_n$,
\item[(2)] $N_2\subset M_{n-1}$,
\item[ ]  \quad \;  $\vdots$
\item[($n$)] $N_n\subset M_1$,
\item[($\kappa$)] $\mathcal A(N_{k+1}) \cap B(N_k)\ne \emptyset$, $N_k\subset D(\mathcal B)$, for $k=1,2,\cdots, n$.
\end{itemize}
\begin{proof}
First we prove the equivalence of the conditions $(1)$ to $(n)$. For each $r=1, 2, \cdots, n-1$, $(r)$ implies $(r+1)$. In fact if
$N_r\subset M_{n-r+1}$, then \eqref{Induction-1}, \eqref{Induction-2} and \eqref{Null-Space} imply that
\begin{eqnarray*}
N_{r+1}&=&\mathcal A^{-1}(\mathcal B(M_r))\subset \mathcal A^{-1}(\mathcal B(M_{n-r+1}))\subset \mathcal A^{-1}(\mathcal A(M_{n-r})+\mathcal B(0))\\
&\subset& \mathcal A^{-1}(\mathcal A(M_{n-r})+\mathcal A(0))= \mathcal A^{-1}[\mathcal A(M_{n-r})+\mathcal A(N(\mathcal A))]\\
&=& \mathcal A^{-1}[\mathcal A(M_{n-r})+N(\mathcal A)] \subset M_{n-r}+N(\mathcal A)+\mathcal A^{-1}(0)\\
&=& M_{n-r}+\mathcal A^{-1}(0) = M_{n-r}+ N_1\subset M_{n-r}+ N_r \subset M_{n-r}+ M_{n-r+1}\\
&=&M_{n-r}.
\end{eqnarray*}
Conversely $(r+1)$ implies $r$. In fact, if $N_{r+1}\subset M_{n-r}$, then
$$
N_r\subset N_{r+1}\subset M_{n-r}=\mathcal B^{-1}(\mathcal A(M_{n-r-1}))
$$
so that each $x\in N_r$ has the property that there exists a $z\in \mathcal B(x)$ such that $z\in \mathcal A(y)$ for some $y\in M_{n-r-1}$. Then $y\in \mathcal A^{-1} (\mathcal B(N_r))=N_{r+1}\subset M_{n-r}$ and so $x\in \mathcal B^{-1}(\mathcal A(M_{n-r}))=M_{n-r+1}$ This proves that $N_r\subset M_{n-r+1}$.

Next we prove that $(n)$ implies $(\kappa)$. So, suppose that $(n)$ is satisfied. Then $N_k\subset N_n\subset M_1=\mathcal B^{-1}(\mathcal AX)\subset D(\mathcal B)$ for $k<n$, so that for each $x\in N_k$, there exists a $z\in \mathcal B(x)$ such that $z\in \mathcal A(y)$ for some $y\in X$. Then $y\in \mathcal A^{-1}(\mathcal B(N_k))=N_{k+1}$ and so $\mathcal A(N_{k+1})\cap \mathcal B(N_k)\ne \emptyset$.
\end{proof}
\end{lemma}
If $N_1\subset M_n$ then $N_1\subset M_{n^\prime}$, for all $n^\prime < n$ since $M_n$ is a non increasing sequence. We denote by $\nu(\mathcal A:\mathcal B)$ the smallest number $n$ for which the condition $N_1\subset M_n$ (or any one of the other equivalent conditions) is not satisfied. We set $\nu(\mathcal A:\mathcal B)=\infty$ if there is no such $n$. This is the case if for example $\mathcal A^{-1}(0)\subset \mathcal B^{-1}(0)$.

\begin{lemma}
Let $X$ and $Y$ be Banach spaces and let $\mathcal A, \mathcal B \in CLR(X,Y)$ with $D(\mathcal A)=D(\mathcal B)=X$. Then
\begin{equation}\label{Adjoint-Sequences}
M_n^\prime \subset (\mathcal B(N_n))^\perp \; {\rm and}\; N_n^\prime \subset (\mathcal A(M_{n-1}))^\perp\; {\rm for}\; n= 1, 2, \dots.
\end{equation}
\end{lemma}
\begin{proof}
First we show that \eqref{Adjoint-Sequences} holds for $n=1$. To begin with, let $y^\prime \in M_1^\prime$ and let $x\in D(\mathcal B)\cap N_1$. Then by definition, $y^\prime \in \mathcal {B^\prime}^{-1}[\mathcal A^\prime (Y^\prime)]$ and $x\in \mathcal A^{-1}(0)\cap D(\mathcal B)$. Hence there exists an element $x^\prime \in \mathcal A^\prime (Y^\prime)\cap R(\mathcal B^\prime)$ such that $(y^\prime, x^\prime)\in G(\mathcal B^\prime)$. Since $x^\prime \in A^\prime (Y^\prime)$, there exists an element $f^\prime \in D(\mathcal A^\prime) \subset Y^\prime$ such that $(f^\prime, x^\prime)\in G(\mathcal A^\prime)$. Since $(x,0)\in G(\mathcal A)$, \eqref{Adjoint} implies that $f^\prime(0)=x^\prime(x)$ so that $x^\prime(x)=0$. So, for $y\in \mathcal B(x)$, $y^\prime(y)=x^\prime(x)=0$, showing the $y^\prime\in [\mathcal B(N_1)]^\perp$.

The second inclusion follows from
$$
N_1^\prime =N(\mathcal A^\prime)=R(\mathcal A)^\perp = [\mathcal A(M_0)]^\perp \; \; ({\rm see\; Lemma}\; \ref{Null-Space-T-0}\; (a)).
$$
We shall therefore assume that \eqref{Adjoint-Sequences} has been proved for $n=k$ and prove it for $n=k+1$. So, let $g^\prime \in M_{k+1}^\prime$ and let $z\in D(\mathcal B)\cap N_{k+1}$. Then $g^\prime \in {\mathcal B^\prime}^{-1}[\mathcal A^\prime (M^\prime_k)]$ and $z\in \mathcal A^{-1}[\mathcal B(N_k)]\cap D(\mathcal B)$. Hence there exists an element $h^\prime \in \mathcal A^\prime (M_k^\prime)$ such that $(g^\prime, h^\prime)\in G(\mathcal B^\prime)$. Since $h^\prime \in \mathcal A^\prime (M_k^\prime)$ it follows that there exists an element $l^\prime \in M_k^\prime$ such that $(l^\prime, h^\prime)\in G(\mathcal A^\prime)$. The fact that $z\in N_{k+1}$ means that there is an element $w\in \mathcal B(N_k)$ such that $(z,w)\in G(\mathcal A)$. This means that $l^\prime (w)=h^\prime(z)$ and $h^\prime (z)=0$ since $l\in [\mathcal B(N_k)]^\perp$. So, for $u\in \mathcal B(z)$, $g^\prime(u)=h^\prime(z)=0$ meaning that $g^\prime \in [\mathcal B(N_{k+1})]^\perp$ and that $M_{k+1}^\prime \subset [\mathcal B(N_{k+1})]^\perp$. This proves the first inclusion in \eqref{Adjoint-Sequences}. The second inclusion can be proved in a similar way.
\end{proof}
\begin{lemma}\label{Contains-Null-Space}
Let $\mathcal A\in CLR (X,Y)$. For every $f^\prime \in N(\mathcal A)^\perp$ there exists $g^\prime \in Y^\prime$ such that $g^\prime (y)=f^\prime (x)$ for all $y\in \mathcal Ax$, and all $x\in D(\mathcal A)$.
\end{lemma}
\begin{proof}
Define a linear functional $g^\prime$ on $Y^\prime$ by setting $g^\prime(y)=f^\prime (x)$ for all $y\in \mathcal A(x)$ and all $x\in D(\mathcal A)$.
Then $g^\prime$ is defined on $R(\mathcal A)$ and is bounded. To show that $g^\prime$ is indeed bounded, we first note that for $y\in \mathcal A(x)$,
\begin{equation}\label{Bounded-Functional}
\vert g^\prime(y)\vert = \vert f^\prime (x)\vert \leq \Vert f\Vert \Vert x\Vert
\end{equation}
and consider the quotient space $\widetilde{X}:=X\slash N(\mathcal A)$. Let $x_1\in \widetilde{x}$. Then $x-x_1=u$ for some $u\in N(\mathcal A)$ so that $f(x)=f(x_1)$. This equality means that $\Vert x\Vert $ in \eqref{Bounded-Functional} can be replaced with $\Vert x_1\Vert$ for any $x_1\in \widetilde{x}$ without changing the inequality. This therefore means that
\begin{eqnarray*}
\vert g^\prime(y)\vert &\leq& \Vert f^\prime\Vert \Vert \widetilde{x}\Vert \\
&\leq& \Vert f^\prime\Vert \gamma(\mathcal A)^{-1}\Vert \mathcal Ax\Vert\\
&=&\Vert f^\prime\Vert \gamma(\mathcal A)^{-1}\Vert  Q_{\mathcal A}y\Vert \\
&\leq& \Vert f^\prime\Vert \gamma(\mathcal A)^{-1}\Vert Q_{\mathcal A}\Vert \Vert y\Vert,
\end{eqnarray*}
that is, $g^\prime$ is bounded on $R(\mathcal A)$. The Hahn-Banach extension theorem implies that $g^\prime$ can be extended to the whole of $Y^\prime$ without changing its bound.
\end{proof}
\begin{rem}\label{Remark-Null}
Lemma \ref{Contains-Null-Space} above implies that $N(\mathcal A)^\perp \subset R(\mathcal A^\prime)$ and that $N(\mathcal A)^\perp = R(\mathcal A^\prime)$ by Lemma \ref{Null-Space-T-0} $(c)$.
\end{rem}

\begin{lemma}\label{Nu}
Let $\mathcal A, \mathcal B\in CLR(X,Y)$ with $D(\mathcal A)=D(\mathcal B)=X$, $R(\mathcal A)$ closed and $\mathcal B$ bounded. If $\mathcal B(0)\subset \mathcal A(0)$ then
\begin{equation}\label{Equality-M}
M_1^\prime = [\mathcal B(N_1)]^\perp.
\end{equation}
\begin{equation}\label{Equality-V}
\nu(\mathcal A^\prime :\mathcal B^\prime)=\nu(\mathcal A :\mathcal B).
\end{equation}
\end{lemma}
\begin{proof}
Let $f^\prime \in [\mathcal B(N_1)]^\perp=(\mathcal B(\mathcal A^{-1}(0)))^\perp$. Since $\mathcal B(0)\subset \mathcal B(N_1)$, Lemma \ref{Norms} $(a)$ together with Lemma \ref{Domains} (a) imply that $f^\prime \in D(\mathcal B^\prime)$. So, let $g^\prime \in \mathcal B^\prime (f^\prime)$, that is, $(f^\prime, g^\prime)\in G(\mathcal B^\prime)$. This means that for $x\in N_1$ and $y\in \mathcal B(x)$, $g^\prime(x)=f^\prime (y)=0$, which shows that $g^\prime\in N_1^\perp=N(\mathcal A)^\perp$ and therefore $g^\prime \in R(\mathcal A^\prime)$  and so $g^\prime \in R(\mathcal A^\prime)$ by Remark \ref{Remark-Null}. It follows that $f^\prime \in {\mathcal B^\prime}^{-1}[\mathcal A^\prime (Y^\prime)]=M_1^\prime$. This shows that $[\mathcal B(N_1)]^\perp \subset M_1^\prime$. Equality \eqref{Equality-M} then follows by \eqref{Adjoint-Sequences}. To prove the second equality, let $v=\nu(\mathcal A:\mathcal B)$. Then $N_1\subset M_n$ for all $n<v$. Since $M_n=\mathcal B^{-1}[\mathcal A(M_{n-1})]$ we see that
\begin{equation}\label{Chain}
\mathcal B(N_1)\subset \mathcal B(M_n)\subset \mathcal A(M_{n-1})+\mathcal B(0)\subset \mathcal A(M_{n-1})+\mathcal A(0)= \mathcal A(M_{n-1}),
\end{equation}
where the last equality follows from the fact that $\mathcal A(0)\subset \mathcal A(M_{n-1})$ and $\mathcal A(M_{n-1})$ is a linear space. We see from \eqref{Chain} that $[\mathcal A(M_{n-1})]^\perp \subset [\mathcal B(N_1)]^\perp$. If then follows from \eqref{Adjoint-Sequences} and \eqref{Equality-M} that $N_n^\prime \subset M_1^\prime$. This means that $v^\prime=(\mathcal A^\prime:\mathcal B^\prime)>n$ and that $v^\prime \geq v$.

To prove the opposite inequality, let $ n<v^\prime$. Then we have $N_1^\prime \subset M_n^\prime$. If follows from Lemma \ref{Null-Space-T-0} $(a)$, \eqref{Null-Space}, and \eqref{Adjoint-Sequences} that $[\mathcal A(X)]^\perp \subset [\mathcal B(N_n)]^\perp$. Since $R(\mathcal A)=\mathcal A(X)$ is closed, this implies that $\mathcal B(N_n) \subset A(X)$. Since $D(\mathcal B)=X$ we see  that $N_n\subset N_n+\mathcal B(0)\subset \mathcal B^{-1}[\mathcal A(X)]=M_1$. This shows that $v>n$ and therefore $v\geq v^\prime$.
\end{proof}

\section{Nullity and Deficiency}
In this section we study the behaviour of the nullity and deficiency for linear relations under some perturbations. For $\mathcal A\in LR(X,Y)$, the nullity $\alpha(\mathcal A)$ and the deficiency $\beta(\mathcal A)$ are defined by
$$
\alpha(\mathcal A):={\rm dim}\; N(\mathcal A) \; \; {\rm and}\; \beta(\mathcal A):={\rm dim}\; Y\slash R(\mathcal A).
$$
\begin{lemma}\label{Equality}{\rm ~\cite[III.7.2]{cross1998multivalued}} \\
Let $\mathcal T$ be a closed linear relation with $\gamma(\mathcal T)>0$. Then $\alpha(\mathcal T^\prime) = \beta(\mathcal T)$.
\end{lemma}

Let $X$ and $Y$ be Banach spaces and let $\mathcal A$ be a closed linear relation with $D(\mathcal A)\subset X$ and $R(\mathcal A)\subset Y$. Let $n\in \left\{ \mathbb N \cap \infty\right\}$ be such that for any $\varepsilon >0$ there exists an $n$-dimensional closed linear subset $N_\varepsilon$ of $N(\mathcal A)$ such that
\begin{equation}\label{Condition}
\Vert \mathcal A(x)\Vert \leq \varepsilon \Vert x\Vert \quad {\rm for\; all}\; x\in N_{\varepsilon}
\end{equation}
while this is not true if $n$ is replaced by a larger number. In such a case we set $\alpha^\prime (\mathcal A):=n$ and define $\beta^\prime(\mathcal A)$ to be
\begin{equation}\label{Beta Prime}
 \beta^\prime(\mathcal A):=\alpha^\prime(\mathcal A^\prime).
\end{equation}

The following two lemmas show that $\alpha^\prime (\mathcal A)$ is defined for every closed linear relation $\mathcal A$.
\begin{lemma}\label{LemmaCondition1}
Assume that for every $\varepsilon>0$ and any closed linear subset $\mathcal M$ of $\mathcal X$ of finite codimension, there is an $x\in \mathcal M \cap D(\mathcal A)$ such that $\Vert x\Vert =1$ and $\Vert \mathcal A(x)\Vert \leq \varepsilon$, then $\alpha^\prime (\mathcal A)=\infty$.
\end{lemma}
\begin{proof}
We have to show that for each $\varepsilon >0$, there exists an infinite dimensional closed linear subset $N_{\varepsilon}\subset D(\mathcal A)$ with the property \eqref{Condition}. First we construct two sequences $x_n\in D(\mathcal A)$ and $f_n\in X'$ such that
\begin{eqnarray}\label{Induction}
&& \Vert x_n\Vert =1, \Vert f_n\Vert =1, f_n(x_n)=1, \nonumber \\
&& f_k(x_n)=0, \; k=1, 2, \cdots , n-1,\\
&& \Vert \mathcal A(x_n)\Vert \leq 3^{-n}\varepsilon,\;  n\in \mathbb N. \nonumber
\end{eqnarray}
For $n=1$, the result holds by ~\cite[III-Corollary 1.24]{kato2013perturbation}. Suppose that $x_n, f_k$ have been constructed for $k=1, 2, \cdots, n-1$. Then $x_n$ and $f_n$ can be constructed in he following way. Let $ M\subset X$ be the collection of all $x\in
X$ such that $f_k(x)=0$, $k=1, 2, \cdots, n-1$. Since $M$ is a closed linear subset of $X$ with finite codimension (dim $M^\perp \leq n-1$ and use codim $M$=dim $M^\perp$), there is an $x_n\in M\cap D(\mathcal A)$ such that $\Vert x_n\Vert =1$ and $\Vert A(x_n)\Vert \leq 3^{-n}\varepsilon$. For this $x_n$, there exists an $f_n\in \mathcal X'$ such that $\Vert f_n\Vert =1$ and $f_n(x_n)=1$ (see ~\cite[III-Corollary 1.24]{kato2013perturbation}. It follows from \eqref{Induction} that the $x_n$ are linearly independent so that
$M_{\varepsilon}^\prime:={\rm span}\;\left\{ x_1, x_2, \cdots \right\}$ is infinite dimensional. Each $x\in M_{\varepsilon}^\prime$ has the form
\begin{equation}\label{form1}
x=\xi_1x_1+\xi_2x_2+\cdots +\xi_nx_n
\end{equation}
for some positive integer $n$. Hence for $k=1, 2, \cdots, n$,
\begin{equation}\label{form2}
f_k(x)=\xi_1f_k{x_1}+\xi_2f_k(x_2)+\cdots +\xi_{k-1}f_k(x_{k-1})+\xi_k.
\end{equation}

We show that the coefficients $\xi_k$ satisfy the inequality
\begin{equation}\label{coefficients}
\vert \xi_k\vert \leq 2^{k-1}\Vert x\Vert, \quad k=1,2, \cdots, n.
\end{equation}
For $k=1$, this is clear from \eqref{Induction} and \eqref{form2}. If we assume that \eqref{coefficients} has been proved for $k<j$, we see from \eqref{form2} that
\begin{eqnarray*}
\vert \xi_j\vert &\leq& \vert f_j(x)\vert+\vert \xi_1\vert \vert f_j(x_1)\vert +\cdots + \vert \xi_{j-1}\vert \vert f_j(x_{j-1})\vert\\
&\leq& \Vert x\Vert + \vert \xi_1\vert +\vert \xi_2\vert +\cdots + \vert \xi_{j-1}\vert \\
&\leq& \Vert x\Vert + \Vert x\Vert + 2\Vert x\Vert + \cdots + 2^{j-2}\Vert x\Vert \\
&=& \Vert x\Vert \left[ 2+2\left(1+2+2^2+\cdots+2^{j-1}\right) \right] \\
&=& 2^{j-1}\Vert x\Vert.
\end{eqnarray*}
It follows from \eqref{Induction}, \eqref{Condition} and \eqref{Condition} that
\begin{eqnarray*}
\Vert \mathcal A(x)\Vert &\leq& \vert\xi_1\vert \Vert \mathcal Ax_1\Vert +\cdots+\vert \xi_n\vert \Vert \mathcal Ax_n\Vert\\
&\leq& \left( \frac{1}{3}+\frac{2}{3^2} + \frac{2^2}{3^3}+\cdots+\frac{2^{n-1}}{3^n} \right)\varepsilon\Vert x\Vert\\
&\leq& \varepsilon\Vert x\Vert.
\end{eqnarray*}
Let $u\in \overline{\mathcal M}_{\varepsilon}^\prime$ and let $\{u_n\}$ be a sequence in $\mathcal M_{\varepsilon}^\prime$ such that $u_n\to u$. The boundedness of $Q_{\mathcal A}\mathcal A$ on $M_{\varepsilon}^\prime$ implies that $\{Q_{\mathcal A}\mathcal A(x_n)\}$ is a Cauchy sequence in $\widetilde{\mathcal Y}:=Y\slash \mathcal A(0)$ and therefore converges, say $Q_{\mathcal A}\mathcal A(x_n)\to \widetilde{v}\in \widetilde{\mathcal Y}$. This means that dist $(x_n-v, \mathcal A(0))\to 0$ as $n\to \infty$, that is, $x_n-v\to z\in \mathcal A(0)$ for some $z\in \mathcal A(0)$. In other words, $x_n\to v+z=w\in \widetilde{v}$. The closedeness of $\mathcal A$ implies that $x\in D(\mathcal A)$ and $w\in \mathcal A(x)$. Hence $Q_{\mathcal A}\mathcal A$ is defined and bounded on the closure of $M_{\varepsilon}^\prime$ with the same bound.
\end{proof}
\begin{lemma}\label{Closed Closed Range}
If $\mathcal A$ is a closed linear relation with closed range(that is, $\gamma(\mathcal A)>0$) then $\alpha^\prime(\mathcal A)=\alpha(\mathcal A)$ and $\beta^\prime(\mathcal A)=\beta(\mathcal A)$.
\end{lemma}
\begin{proof}
By Lemma \ref{Norms}, $\gamma(\mathcal A)>0$ implies $\gamma(\mathcal A^\prime)>0$ while Lemma \ref{Equality} implies that  $\alpha(\mathcal A^\prime) = \beta(\mathcal A)$. In view of \eqref{Beta Prime}, it is enough to show that $\alpha^\prime(\mathcal A)=\alpha(\mathcal A)$. It is clear that $\alpha^\prime(\mathcal A)\geq\alpha(\mathcal A)$. Now suppose that there exists a closed linear manifold $N_\varepsilon$ with ${\rm dim} N_\varepsilon > \alpha(\mathcal A)= {\rm dim} N(\mathcal A)$ and with property \eqref{Condition}. Pick $x\in N_{\varepsilon}$ such that $\Vert \widetilde x\Vert =\Vert x\Vert=1$ where $\widetilde x\in \widetilde{X}=: X\slash N(\mathcal A)$ (this is possible by ~\cite[Lemma 241]{kato1958perturbation}). For this $x$, $\Vert \mathcal A(x)\Vert \geq \gamma(\mathcal A)$ on the one hand and $\Vert \mathcal A(x)\Vert \leq \varepsilon$ on the other hand, leading to the inequality $\gamma(\mathcal A)\leq \varepsilon$. In other words, there is no $N_\varepsilon$ with ${\rm dim} N_\varepsilon > \alpha(\mathcal A)= {\rm dim} N(\mathcal A)$ for $\varepsilon < \gamma(\mathcal A)$. This proves that $\alpha^\prime(\mathcal A)\leq \alpha(\mathcal A)$ and that $\alpha^\prime(\mathcal A)= \alpha(\mathcal A)$. The second equality follows from \eqref{Beta Prime} and Lemma \ref{Equality}.
\end{proof}
%
\begin{lemma}\label{Infinity}
Let $\mathcal T\in CLR(X)$ with non closed range (that is, $\gamma(\mathcal T)=0$), then
\begin{equation}\label{Nonclosed}
\alpha^\prime(\mathcal T)=\infty.
\end{equation}
\end{lemma}
\begin{proof}
Let $M$ be any closed linear manifold of $X$ with finite codimension and let $Q_{\mathcal T}$ be denoted by $Q$. Consider the mapping $T:X\slash M\to Q\mathcal T(X)\slash Q\mathcal T(M)$ defines by setting $T(\widetilde x)=\widetilde{Q\mathcal T(x)}$. Then $T$ is clearly well defined and linear. It is well defined since
$$
T(\widetilde{x+v})=\widetilde{Q\mathcal T(x+v)}=\widetilde{Q\mathcal T(x)+Q\mathcal T(v)} = \widetilde{Q\mathcal T(x)}=T\widetilde{x}
$$
for any $v\in M$. If follows that $Q\mathcal T(X)\slash Q\mathcal T(M)$ is a finite dimensional space since $M$ has finite codimension. ~\cite[III-Lemma 1.9]{kato2013perturbation} implies that $Q\mathcal T(X)$ is a closed subset of $\widehat{Y}:=Y\slash \mathcal T(0)$ if $Q\mathcal T(M)$ is a closed subspace of the same space. This would mean that $\mathcal T(X)$ is a closed subset of $Y$. To see why this is true, let $\{y_n\}$ be a convergent sequence in $\mathcal T(X)$ with $y_n\to y\in Y$. Then $\{Qy_n\}$ is a Cauchy sequence in $\widehat{Y}$ and therefore converges to some point $\widehat z\in Q\mathcal T(X)$. In other words, $y_n-z\to w\in \mathcal T(0)$, so that $y_n\to z+w\in \widehat z$. The uniqueness of the limit implies that $y=z+w\in \widehat{z}$ and that $y\in R(\mathcal T)$ since $z\in R(\mathcal T)$ and every coset that contains and element of $R(\mathcal T)$ consists entirely of elements of $R(\mathcal T)$. Next we show that if $\mathcal T(M)$ is closed then $Q\mathcal T(M)$ is closed. So, assume that $\mathcal T(M)$ is closed and let $\{ \widehat z\}$ be a sequence in $Q\mathcal T(M)$ that converges to an element $\widehat{z}\in \widehat{Y}$. Then $z_n-z\to v\in \mathcal T(0)$ and so $z_n\to z+v\in \widehat{z}$. The closedness of $\mathcal T(M)$ implies that $z+v\in
\mathcal T(M)$ and that $\widehat{z}\in Q\mathcal T(M)$.

The contradiction that $\mathcal T(X)$ is both open and closed means that $Q\mathcal T(M)$ is not closed and that $\mathcal T(M)$ is not closed and therefore $\gamma\left( \mathcal T_M\right)=0$. Hence there exists, for any $\varepsilon >0$, an $x\in M\cap D(\mathcal T)$ such that $\Vert x\Vert =1$ and $\Vert \mathcal T(x)\Vert \leq \varepsilon \Vert \widetilde x\Vert \leq \varepsilon\Vert x\Vert=\varepsilon$  where $\widetilde x\in \widetilde X=X\slash N(\mathcal T)$. This shows that the conditions of Lemma \ref{LemmaCondition1} are satisfied and therefore $\alpha^\prime(\mathcal T)=\infty$.
\end{proof}
\begin{theorem}\label{Complete-1}
Let $X$ and $Y$ be Banach spaces and let $\mathcal A$ be a closed linear relation with $D(\mathcal A)\subset X$, having closed range $R(\mathcal A)\subset Y$, and with $\alpha(\mathcal A)$ finite. Let $\mathcal B$ be a closed bounded linear relation such that $D(\mathcal B)\supset D(\mathcal A)$, $\mathcal B(0)\subset \mathcal A(0)$ and
\begin{equation}\label{Bounded}
\Vert \mathcal B\Vert < \gamma(\mathcal A).
\end{equation}
Then the linear relation $\mathcal A+\mathcal B$ is closed and has closed range. Moreover,
\begin{equation}\label{Inequalities}
\alpha(\mathcal A+\mathcal B)\leq \alpha(\mathcal A), \quad \beta(\mathcal A +\mathcal B)\leq \beta(\mathcal A).
\end{equation}
\end{theorem}
\begin{proof}
Let $\left\{ x_n\right\}$ be a sequence in $D(\mathcal A)$ such that $x_n\to x\in X$ and let $\left\{ y_n\right\}$ be a sequence in $R(\mathcal A+\mathcal B)$ such that $y_n\to y\in Y$, where $y_n=u_n+v_n$ with $u_n\in \mathcal A(x_n)$ and $v_n\in \mathcal B(x_n)$ for each $n\in \mathbb N$. In other words,
\begin{equation}\label{Limit}
u_n+v_n\to y.
\end{equation}
Note that \eqref{Bounded} implies that $\{Q_{\mathcal B}\mathcal B(x_n)\}$ is a Cauchy sequence in $\widetilde{Y}:=Y\slash \mathcal B(0)$ and therefore converges to a point of $\widetilde{Y}$, say $Q_{\mathcal B}\mathcal B(x_n)\to \widetilde v\in \widetilde{Y}$. Hence ${\rm dist}(v_n-v, \mathcal B(0))\to 0$ as $n\to 0$, that is, $v_n-v\to z$ for some $z\in \mathcal B(0)$. Hence $v_n\to v+z\in \widetilde v$. The closedness of $\mathcal B$ implies that $x\in D(\mathcal B)$ and $v+z\in \mathcal B(x)$. Hence $y=y-v-z+(v+z)\in \mathcal A(x)+\mathcal B(x)$ and so $\mathcal A+\mathcal B$ is closed.



To complete the proof, it is enough to show that
\begin{equation}\label{Enough}
\alpha^\prime(\mathcal A+\mathcal B)\leq \alpha(\mathcal A)\; {\rm and\; that}\; \beta^\prime(\mathcal A+\mathcal B)\leq \beta^\prime(\mathcal A)
\end{equation}
and then apply Lemma \eqref{Infinity} to conclude that $\mathcal A+\mathcal B$ has closed range and Lemma \ref{Closed Closed Range} to establish the inequalities in the theorem since $\alpha^\prime(\mathcal A+\mathcal B)\geq \alpha(\mathcal A+\mathcal B)$ by definition and $\beta^\prime(\mathcal A+\mathcal B)\geq \alpha(\mathcal A+\mathcal B)$ by \eqref{Beta Prime} and Lemma \ref{Equality}.

To prove \eqref{Enough}, suppose that for a given $\varepsilon>0$ there exists a closed linear manifold $N_\varepsilon\subset D(\mathcal A+\mathcal B)=D(\mathcal A)$ such that
\begin{equation}\label{Subspace}
\Vert (\mathcal A+\mathcal B)(x)\Vert \leq \varepsilon\Vert x\Vert \; {\rm for\; every}\; x\in N_\varepsilon.
\end{equation}
It then follows form \eqref{Subspace} and Lemma \ref{Norm Difference} that
\begin{eqnarray}\label{Result}
(\Vert \mathcal B\Vert+\varepsilon)\Vert x\Vert &\geq& \Vert \mathcal B(x)\Vert +\Vert (\mathcal A+\mathcal B)(x)\Vert \geq \Vert \mathcal B(x)\Vert +(\Vert \mathcal Ax\Vert -\Vert \mathcal B(x)\Vert)\nonumber \\
&\geq& \Vert \mathcal A(x)\Vert \geq  \gamma(\mathcal A)\Vert \widetilde{x}\Vert
\end{eqnarray}
where $\widetilde{x}\in \widetilde{X}:=X\slash N(\mathcal A)$. If we pick $\varepsilon$ such that $0<\varepsilon<\gamma(\mathcal A)-\Vert \mathcal B\Vert$ we see from \eqref{Result} that $\Vert \widetilde{x}\Vert < \Vert x\Vert $ for all non-zero $x\in D(\mathcal A)$. It therefore follows from Lemmma \ref{Dimension} that
$$
{\rm dim}\; N_\varepsilon \leq {\rm dim}\; N(\mathcal A) = \alpha(\mathcal A),
$$
which means that $\alpha^\prime(\mathcal A+\mathcal B)\leq \alpha (\mathcal A)$.

To prove the second inequality, we note that Lemma \ref{Norms} together with Lemma \ref{Domains} imply that $\Vert \mathcal B^\prime\Vert =\Vert \mathcal B\Vert$,  $\gamma(\mathcal A^\prime)=\gamma(\mathcal A)$, and $(\mathcal A+\mathcal B)^\prime = \mathcal A^\prime +\mathcal B^\prime$. It therefore follows that $\Vert \mathcal B^\prime\Vert \leq \gamma(\mathcal A^\prime)$. Applying what has been proved above to the pair $\mathcal A^\prime, \; \mathcal B^\prime$, we see that
$$
\beta^\prime(\mathcal A+\mathcal B)=\alpha^\prime((\mathcal A+\mathcal B)^\prime)=\alpha^\prime(\mathcal A^\prime +\mathcal B^\prime)
\leq \alpha(\mathcal A^\prime)=\beta(\mathcal A),
$$
where the last equality follows from Lemma \ref{Equality}.
\end{proof}

\begin{lemma}\label{Banach Space}
Let $X$ and $Y$ be Banach spaces and let $\mathcal T$ be a closed linear relation with $D(\mathcal T)\subset X$ and $R(\mathcal T)\subset Y$. Set
\begin{equation}\label{New norm}
\Vert x\Vert_{D(\mathcal T)}:= \Vert x\Vert + \Vert \mathcal T(x)\Vert, \quad x\in D(\mathcal T).
\end{equation}
Then $D(\mathcal T)$ becomes a Banach space if $\Vert \cdot \Vert_{D(\mathcal T)}$ is chosen as the norm.
\end{lemma}
\begin{proof}
That $\Vert \cdot \Vert_{D(\mathcal T)}$ defines a norm on $D(\mathcal T)$ is clear. To prove completeness, assume that $\left\{x_n\right\}$ is a Cauchy sequence in $D(\mathcal T)$. Then $\left\{x_n\right\}$ and $\left\{Q_{\mathcal T}\mathcal T(x_n)\right\}$ are Cauchy sequences in $X$ and $\widetilde{Y}=Y\slash \mathcal T(0)$ respectively and therefore converge, say, $x_n\to x\in X$ and $Q_{\mathcal T}\mathcal T(x_n)\to \widetilde{u}\in \widetilde{Y}$. Let $u_n\in \mathcal T(x_n)$ for each $n\in \mathbb N$. Then $\widetilde{u}_n\to \widetilde{u}$ and so ${\rm dist}\;(u_n-u, \mathcal T(0))\to 0$ as $n\to \infty$, that is, $u_n-u\to v\in \mathcal T(0)$. We therefore see that $u_n\to u+v=s\in \widetilde{u}$. The closedness of $\mathcal T$ implies that $x\in D(\mathcal T)$ and that $s\in \mathcal T(x)$. Now,
\begin{eqnarray*}
\Vert x_n-x\Vert_{D(\mathcal T)}&=&\Vert x_n-x\Vert+\Vert Q_{\mathcal T}\mathcal T(x_n-x)\Vert=\Vert x_n-x\Vert+
\Vert Q_{\mathcal T}u_n-Q_{\mathcal T}s)\Vert\\
&=& \Vert x_n-x\Vert+\Vert \widetilde{u}_n-\widetilde{u}\Vert\to 0 \; {\rm as}\; n\to \infty.
\end{eqnarray*}
This shows that $D(\mathcal T)$ is complete.
\end{proof}
Let $X$ and $Y$ be Banach spaces and let $\mathcal{A}, \mathcal{B}\in CLR(X, Y)$ be such that $D(\mathcal A)\subset D(\mathcal B)$ and $\mathcal B(0)\subset \mathcal A(0)$. In the following theorem we write $\Vert \mathcal B(x)\Vert_{\mathcal A}$ to mean the quantity $\Vert Q_{\mathcal A}\mathcal{B}(x)\Vert$. The quantities $\Vert \mathcal A(x)\Vert_{\mathcal A}$ and $\Vert \mathcal B(x)\Vert_{\mathcal B}$ are defined in a similar way

\begin{theorem}\label{Complete}
Let $X$ and $Y$ be Banach spaces and let $\mathcal A$ be a closed linear relation with $D(\mathcal A)\subset X$ and with closed range $R(\mathcal A)\subset Y$. Let $\mathcal B$ be a closed linear relation such that $D(\mathcal A) \subset D(\mathcal B)\subset X$, $R(\mathcal B)\subset Y$, $\mathcal B(0)\subset \mathcal A(0)$, and
\begin{equation}\label{Equation 1}
\Vert \mathcal B(x)\Vert_{\mathcal B} \leq \sigma\Vert x\Vert + \tau\Vert \mathcal A(x)\Vert_{\mathcal A}, \; \forall x\in D(\mathcal A),
\end{equation}
where $\sigma$ and $\tau$ are non-negative constants such that
\begin{equation}\label{Equation 2}
\sigma + \tau \gamma(\mathcal A)<\gamma(\mathcal A).
\end{equation} Then the linear relation $\mathcal A + \mathcal B$ is closed and has closed range. If $\alpha (\mathcal A)<\infty$ then
\begin{equation}\label{Equation 3}
\alpha(\mathcal A +\mathcal B)\leq \alpha(\mathcal A), \; {\rm and}\; \beta(\mathcal A +\mathcal B)\leq \beta(\mathcal A).
\end{equation}
\end{theorem}
\begin{proof}
Let $\left\{ x_n\right\}$ be a sequence in $D(\mathcal A)$ such that $x_n\to x\in X$ and let $\left\{ y_n\right\}$ be a sequence in $R(\mathcal A+\mathcal B)$ such that $y_n\to y\in Y$, where $y_n=u_n+v_n$ with $u_n\in \mathcal A(x_n)$ and $v_n\in \mathcal B(x_n)$ for each $n\in \mathbb N$. Note that \eqref{Equation 1} implies that
\begin{equation}\label{Inequality 1}
\Vert \mathcal A(x)\Vert_{\mathcal A} -\Vert \mathcal B(x)\Vert_{\mathcal B} \geq (1-\tau)\Vert \mathcal A(x)\Vert_{\mathcal A}-\sigma \Vert x\Vert.
\end{equation}
Since $\Vert \mathcal B(x)\Vert_{\mathcal B} =\Vert Q_{\mathcal B}\mathcal B(x)\Vert_{\mathcal B} \geq \Vert Q_{\mathcal A}\mathcal Bx\Vert_{\mathcal A}$
we see that
\begin{equation}\label{Inequality 2}
\Vert Q_{\mathcal A}\mathcal A(x)\Vert_{\mathcal A} -\Vert Q_{\mathcal A}\mathcal B(x)\Vert_{\mathcal A} \geq (1-\tau)\Vert \mathcal A(x)\Vert-\sigma \Vert x\Vert
\end{equation}
and that
\begin{equation}\label{Inequality 3}
\Vert Q_{\mathcal A} \mathcal Ax + Q_{\mathcal A}\mathcal Bx\Vert_{\mathcal A} \geq (1-\tau)\Vert Q_{\mathcal A}\mathcal Ax\Vert-\sigma \Vert x\Vert.
\end{equation}
Inequality \eqref{Inequality 3} and the linearity of $Q_{\mathcal A}$ implies that
\begin{equation}\label{Inequality 4}
\Vert Q_{\mathcal A} (u_n+v_n)\Vert_{\mathcal A} \geq (1-\tau)\Vert Q_{\mathcal A}u_n\Vert-\sigma \Vert x_n\Vert
\end{equation}
so that
\begin{equation}\label{Inequality 5}
\Vert y_n\Vert=\Vert u_n+v_n \Vert \geq (1-\tau)\Vert Q_{\mathcal A}u_n\Vert-\sigma \Vert x_n\Vert.
\end{equation}

It therefore follows that for $m,n\in \mathbb N$,
\begin{equation}\label{Inequality 6}
\Vert y_n-y_m\Vert \geq (1-\tau)\Vert  Q_{\mathcal A}u_n-Q_{\mathcal A}u_m)\Vert-\sigma \Vert x_n-x_m\Vert.
\end{equation}
Since $1-\tau>0$ by \eqref{Equation 2} and both $\{ x_n\}$ and $\{ y_n\}$ are Cauchy sequences, it follows by \eqref{Inequality 6} that
$\left\{ Q_{\mathcal A}u_n\right\}$ is a Cauchy sequence and therefore converges, say,
\begin{equation}\label{convergence}
\widetilde{u}_n\to \widetilde{u},
\end{equation}
where we denote $Q_{\mathcal A}u_n$ by $\widetilde{u}_n$ in $Y\slash \mathcal A(0)$. The convergence in \eqref{convergence} implies that ${\rm dist}\; \left(u_n-u, \mathcal A(0)\right)\to 0$ as $n\to \infty$. This means that $u_n-u$ converges to an element of $\overline{\mathcal A(0)}=\mathcal A(0)$, say $u_n-u\to z\in \mathcal{A}(0)$. This means that $u_n\to z-u = s$. The closedness of $\mathcal A$ implies that $x\in D(\mathcal A)$ and $s\in \mathcal A(x)$. Since $u_n\to s$, we see that $Q_{\mathcal A}u_n=Q_{\mathcal A}s$. Applying \eqref{Equation 1} to $x_n-x$ we see that $Q_{\mathcal B}\mathcal B(x_n)\to Q_{\mathcal B}\mathcal B(x)$, that is, ${\rm dist}\;\left(v_n-v, \mathcal B(0)\right)\to 0$ as $n\to \infty$, $v\in \mathcal Bx$. This shows that $v_n-v$ converges to an element say $w$ of $\mathcal B(0)$, that is, $v_n\to w-v=r\in \mathcal B(x)$ since $\mathcal B(x)=\mathcal B(0)+v$. Hence $y=s+r\in (\mathcal A+\mathcal B)(x)$, showing that $\mathcal A+\mathcal B$ is closed.

We introduce a norm on $D(\mathcal A)$ by
\begin{equation}\label{New-Norm}
\Vert x\Vert_{\breve{D}}:=(\sigma +\varepsilon)\Vert x\Vert +(\tau +\varepsilon\Vert \mathcal A(x)\Vert \geq \varepsilon\Vert x\Vert,
\end{equation}
for some arbitrary but fixed positive constant $\varepsilon$. Note that the  space $D(\mathcal A)$ becomes a Banach space by Lemma \ref{Banach Space}, which we denote by $\breve{D}$. We now regard $\mathcal A$ and $\mathcal B$ as linear relations with $D(\mathcal A)=D(\mathcal B)=\breve{D}$ and denote them by $\breve{\mathcal A}$ and $\breve{\mathcal B}$ respectively. Since
$\Vert x\Vert_{\breve{D}}= (\sigma +\varepsilon)\Vert x\Vert +(\tau +\varepsilon\Vert \mathcal Ax\Vert >\sigma \Vert x\Vert + \tau\Vert \mathcal Ax\Vert \geq \Vert Bx\Vert$ for every $x\in \breve{D}$ and $\Vert \breve{\mathcal B}\Vert:= \underset{x\in B_{\breve{D}}}{\sup}\Vert \breve{\mathcal B}x\Vert$, we see that $\Vert \breve{\mathcal B}\Vert \leq 1$. From $\Vert \mathcal Ax\Vert \leq (\tau =\varepsilon)^{-1}\Vert x\Vert_{\breve{D}}$ and the definition of $\Vert \breve{\mathcal A}\Vert$, we also see that $\Vert \breve{\mathcal A}\Vert \leq (\tau +\varepsilon)^{-1}$.

It is clear that $R(\breve{\mathcal A})=R(\mathcal A)$ is closed and that
\begin{eqnarray}\label{Same}
\alpha(\breve{\mathcal A})=\alpha(\mathcal A), && \beta(\breve{\mathcal A})=\beta(\mathcal A), \nonumber \\
\alpha(\breve{\mathcal A}+ \breve{\mathcal B})=\alpha(\mathcal A + \mathcal B), &&
\beta(\breve{\mathcal A}+ \breve{\mathcal B})=\beta(\mathcal A + \mathcal B)
\end{eqnarray}
Please note that $\gamma(\breve{\mathcal A})=\gamma(\mathcal A)$ if $\gamma(\mathcal A)=\infty$. In order to relate $\gamma(\breve{\mathcal A})$ to $\gamma(\mathcal A)$ in the other case, we recall that in this case,
$$
\gamma(\breve{\mathcal A})=\inf\left\{ \frac{\Vert \breve{\mathcal A}(x)\Vert}{\Vert \widetilde{x}\Vert_{\breve{D}}}:x\in \breve{D}, x\notin N(\breve{\mathcal A})\right\}
= \inf\left\{ \frac{\Vert \mathcal A(x)\Vert}{\Vert \widetilde{x}\Vert_{\breve{D}}}:x\in \breve{D}, x\notin N(\grave{\mathcal A})\right\}
$$
where $\widetilde{x}\in \widetilde{X}:= X\slash N(\mathcal A)$.

But
\begin{eqnarray*}
\Vert \widetilde{x}\Vert_{\breve{D}} &=& \underset{z\in N(\mathcal A)}{\inf}\Vert x-z\Vert_{\breve{D}}\\
&=&\underset{z\in N(\mathcal A)}{\inf}[(\sigma +\varepsilon)\Vert x-z\Vert +(\tau +\varepsilon)\Vert \mathcal A(x-z)\Vert]\\
&=&(\sigma +\varepsilon)\Vert \widetilde x\Vert +(\tau +\varepsilon)\Vert \mathcal A(x)\Vert
\end{eqnarray*}
where we have used the linearily of the natural quotient map and the fact that $\mathcal A(z)=\mathcal A(0)$.

Hence
\begin{eqnarray*}
\gamma(\breve{\mathcal A}) &=& \inf\left\{
\frac{\Vert A(x)\Vert}{(\sigma +\varepsilon)\Vert \widetilde x\Vert +(\tau +\varepsilon)\Vert \mathcal Ax\Vert}:
x\in D(\mathcal A), x\notin N(\mathcal A)\right\}\\
&=& \frac{\gamma(\mathcal A)}{(\sigma +\varepsilon) +(\tau +\varepsilon)\gamma(\mathcal A)},
\end{eqnarray*}
where we have used the fact that $f(t)=\frac{t}{\alpha +t}$ is an increasing function for any constant $\alpha$.

In view of \eqref{Equation 2}, we can make $\gamma{(\breve{\mathcal A})}>1$ by choosing $\varepsilon$ small enough. Since $\Vert \breve{\mathcal B}\Vert \leq 1$, we can apply Theorem \ref{Complete-1} to the pair $\grave{\mathcal A},\; \breve{\mathcal B}$ with the result that $R(\breve{\mathcal A}+\breve{\mathcal B})=R(\mathcal A+\mathcal B)$ is closed and \eqref{Inequalities} holds with $\mathcal A, \; \mathcal B$ replaced with $\breve{\mathcal A},\; \breve{\mathcal B}$. The result then follows by \eqref{Same}.
\end{proof}

\section{Stability Theorms}
Consider an eigenvalue problem of the form
\begin{equation}\label{Eigenvalue Problem}
Ax=\lambda B
\end{equation}
where $A$ and $B$ are linear operators from $X$ to $Y$ and the associated
\begin{equation}\label{Eigenvalue-Problem}
A^*f^\prime = \lambda B^*f^\prime
\end{equation}
where the adjoints $A^*$ and $B^*$ exist. The null space $N(A-\lambda B)$ of the linear operator $A-\lambda B$ is the solution set of the eigenvalue problem \eqref{Eigenvalue Problem}. Similarly, $N(A^*-\lambda B^*)=R(A-\lambda B)^\perp$ is the solution set of the eigenvalue problem \eqref{Eigenvalue-Problem}. In studying the above eigenvalue problems, one therefore gets interested in the behaviour of $N(A-\lambda B)$ and $N(A^*-\lambda B^*)$.

In the setting of linear relations, the eigenvalue problems \eqref{Eigenvalue Problem} and \eqref{Eigenvalue-Problem} can be formulated as
\begin{equation}\label{Eigenvalue}
\mathcal A(x) \cap \lambda\mathcal B(x)\ne \emptyset
\end{equation}
and
\begin{equation}\label{Eigen-value}
\mathcal A^\prime (x^\prime) \cap \lambda\mathcal B^\prime (x^\prime)\ne \emptyset
\end{equation}
where $\mathcal A,\; \mathcal B\in LR(X,Y)$. Conditions \eqref{Eigenvalue} and \eqref{Eigenvalue-Problem} are equivalent to
\begin{equation}\label{Eigenvalue Equivalent}
(\mathcal A-\lambda \mathcal B)(x)=(\mathcal A-\lambda \mathcal B)(0)
\end{equation}
and
\begin{equation}\label{Eigenvalue-Equivalent}
(\mathcal A^\prime -\lambda \mathcal B^\prime)(x^\prime)=(\mathcal A^\prime-\lambda \mathcal B^\prime)(0)
\end{equation}
respectively.

As before, the solution sets of \eqref{Eigenvalue Equivalent} and \eqref{Eigenvalue-Equivalent} are $N(\mathcal A-\lambda \mathcal B)$ and $N(\mathcal A^\prime -\lambda \mathcal B^\prime)=R(\mathcal A-\lambda\mathcal B)^\perp$ respectively. In this last section we study the stability of the dimensions of the null spaces of $\mathcal A-\lambda\mathcal B$ and $\mathcal A^\prime -\lambda \mathcal B^\prime$ as $\lambda$ varies in some specified subset of the complex plane. This is considered in the following theorems.

\begin{theorem}
Let $X$ and $Y$ be Banach spaces and let $\mathcal A, \mathcal B\in CLR(X,Y)$ be such that $\mathcal A$ has closed range,  $D(\mathcal B)\supset D(\mathcal A)$, $\mathcal B(0)\subset \mathcal A(0)$, and
\begin{equation}\label{Norm-Again-1}
\Vert \mathcal B(x)\Vert \leq \sigma \Vert x\Vert +\tau \Vert \mathcal A(x)\Vert \;\; {\rm for\; every}\; x\in D(\mathcal A)
\end{equation}
where $\sigma$ and $\tau$ are non-negative constants.
Then $\mathcal A-\lambda \mathcal B$ is closed for $\vert \lambda \vert <\frac{\gamma(\mathcal A)}{\sigma+\tau\gamma(\mathcal A)}$ and if $R(\mathcal A)\setminus \mathcal A(0)\ne \emptyset$,  then $\gamma (\mathcal A-\lambda \mathcal B)<\infty$ for $\vert \lambda \vert \geq \frac{\gamma(\mathcal A)}{\sigma+\tau\gamma(\mathcal A)}$.
\end{theorem}

\begin{proof}
If follows from Theorem \ref{Complete} that $\mathcal A-\lambda \mathcal B$ is closed if $\vert \lambda \vert < \frac{\gamma(\mathcal A)}{\sigma+\tau\gamma(\mathcal A)}$.

If $\gamma(\mathcal A-\lambda\mathcal B)=\infty$ then $(\mathcal A-\lambda\mathcal B)(x)=(\mathcal A-\lambda\mathcal B)(0) =\mathcal A(0)$. The fact that $(\mathcal A-\lambda \mathcal B)(x)=(\mathcal A-\lambda\mathcal B)(0)$ for every $x\in D(\mathcal A-\lambda\mathcal B) = D(\mathcal A)$ implies that $\mathcal A(x)\cap \lambda\mathcal B(x)\neq \emptyset$ for every $x\in D(\mathcal A)$. Since $\mathcal B(0)\subset \mathcal A(0)$, it follows that $\Vert \mathcal  A(x)\Vert \leq \Vert \lambda\mathcal B(x)\Vert$ for every $x\in D(\mathcal A)$ and therefore
$$
\Vert \mathcal A(x)\Vert \leq \vert \lambda\vert \Vert \mathcal B(x)\Vert \leq \vert \lambda\vert(\sigma\Vert x\Vert +\tau\Vert \mathcal A(x)\Vert)
$$
so that
\begin{equation}\label{Range Not}
(1-\vert \lambda\Vert \tau)\Vert \mathcal A(x)\Vert \leq \sigma \vert \lambda\vert\Vert x\Vert.
\end{equation}
Since $R(\mathcal A)\ne \mathcal A(0)$, we see that there exists at least one $\widetilde{x}$ in $\widetilde{X}=X\slash N(\mathcal A)$ with $\widetilde{x}\ne 0$. Inequality \eqref{Range Not} implies that
\begin{equation}\label{Not Empty}
\gamma(\mathcal A)\Vert \widetilde{x}\Vert \leq \Vert \mathcal A(x)\Vert \leq \sigma\vert \lambda\vert \Vert x\Vert/(1-\vert \lambda\vert\tau).
\end{equation}
Since $x$ can vary freely in $\widetilde{x}$, we conclude that $\gamma(\mathcal A)\leq \sigma\vert \lambda\vert/(1-\vert \lambda\vert\tau)$ and that $\vert \lambda\vert \geq \frac{\gamma(\mathcal A)}{\sigma +\tau\gamma(\mathcal A)}$.
\end{proof}

\begin{theorem}\label{Stabilty in Lambda-1}
Let $X$ and $Y$ be Banach spaces and let $\mathcal A, \mathcal B\in CLR(X,Y)$ be such that $\mathcal A$ has closed range, $D(\mathcal B)\supset D(\mathcal A)$, $\mathcal B(0)\subset \mathcal T(0)$, and
\begin{equation}\label{Norm-Again}
\Vert \mathcal B(x)\Vert \leq \sigma \Vert x\Vert +\tau \Vert \mathcal A(x)\Vert \;\; {\rm for\; every}\; x\in D(\mathcal A),
\end{equation}
where $\sigma$ and $\tau$ are non negative constants. If $\nu(\mathcal A:\mathcal B)=\infty$ then
\begin{equation}\label{Finishing}
\delta(N(\mathcal A), N(\mathcal A-\lambda\mathcal B))\leq \frac{\sigma\vert \lambda\vert}{\gamma(\mathcal A)-\vert \lambda\vert(\sigma +\tau\gamma(\mathcal A))}.
\end{equation}
\end{theorem}
\begin{proof}
Let $N_k$ be as defined in \eqref{Inductively-2} and consider a sequence $z_k$ with the following properties:
\begin{eqnarray}\label{Sequence}
&& z_k\in N_k,\quad \mathcal A(z_{k+1})\cap \mathcal B(z_k)\ne \emptyset\\
&& \xi\Vert z_{k+1}\Vert \leq \Vert \mathcal A(z_{k+1})\Vert, \quad k=1, 2, \cdots,\nonumber
\end{eqnarray}
where $\xi$ is a positive constant. We show that for each $z\in N(\mathcal A)$ and $\xi <\gamma(\mathcal A)$, there is a sequence $z_k$ that satisfies \eqref{Sequence} such that $z=z_1$. We set $z=z_1$ and construct $z_k$ by induction. Suppose $z_1, z_2, \dots x_k$ have been constructed with properties \eqref{Sequence}. Since $z_k\in N_k\subset M_1=\mathcal B^{-1}(\mathcal A(X))$, there exists a $z_{k+1}\in D(\mathcal A)$ such that
$\mathcal A(z_{k+1})\cap \mathcal B(z_k)\ne \emptyset$. Since $\gamma(\mathcal A)\Vert \widetilde{z}_{k+1}\Vert \leq \Vert \mathcal A(z_{k+1})\Vert$ and $z_{k+1}$ can be replaced by any other element of $\widetilde{z}_{k+1}$ we can choose $z_{k+1}$ such that
$\xi \Vert z_{k+1}\Vert \leq \Vert \mathcal A(z_{k+1})\Vert$. Since $\mathcal A(z_{k+1})\cap \mathcal B(z_k)\ne \emptyset$, we see that
$z_{k+1}\in \mathcal A^{-1}(\mathcal B(N_n))=N_{k+1}$. This completes the induction process.

Since $\mathcal A(z_{k+1})\cap \mathcal B(z_k)\ne \emptyset$ and $\mathcal A(0) \supset \mathcal B(0)$, we see that
\begin{equation}\label{Derive}
\Vert \mathcal A(z_{k+1})\Vert \leq \Vert \mathcal B(z_k)\Vert\leq \sigma\Vert z_k\Vert +\tau\Vert A(z_k)\Vert.
\end{equation}
For $k=1$, \eqref{Derive} gives $\Vert \mathcal A(z_2)\Vert \leq \Vert \mathcal B(z_1)\Vert \leq \sigma\Vert z_1\Vert$ since $z_1\in N(\mathcal A)$.
For $k\ge 2$, \eqref{Sequence} implies that
\begin{eqnarray}\label{Bound-1}
\Vert \mathcal A(z_{k+1})\Vert &\leq& \Vert \mathcal B(z_k)\Vert \leq \sigma\Vert z_k\Vert +\tau\Vert \mathcal A(z_k)\Vert \leq (\sigma\xi^{-1}+\tau)\Vert \mathcal A(z_k)\Vert\\
&\leq& (\sigma\xi^{-1}+\tau)^2\Vert \mathcal A(z_{k-1})\Vert\leq \cdots \leq (\sigma\xi^{-1}+\tau)^{k-1}\Vert \mathcal A(z_2)\Vert \nonumber\\
&=& \xi^{-(k-1)}(\sigma +\xi\tau)^{k-1}\Vert \mathcal A(z_2)\Vert \leq \sigma\xi^{-(k-1)}(\sigma +\xi\tau)^{k-1}\Vert z_1\Vert. \nonumber
\end{eqnarray}
We also see from \eqref{Sequence} and \eqref{Bound-1} that
\begin{equation}\label{Bound-2}
\Vert z_{k+1}\Vert \leq \sigma\xi^{-k}(\sigma +\xi\tau)^{k-1}\Vert z_1\Vert, \quad k=1,2, \dots.
\end{equation}
The bounds in \eqref{Bound-1} and \eqref{Bound-2} imply that the series
$$
u(\lambda)=\sum_{k=1}^\infty \lambda^{k-1}z_k, \quad \lambda(\mathcal A)=\sum_{k=1}^\infty \lambda^kQ_{\mathcal A}\mathcal A(z_{k+1}), \; \; \lambda(\mathcal B)=\sum_{k=1}^\infty \lambda^{k-1}Q_{\mathcal B}\mathcal B(z_k)
$$
and
$$
\lambda(\mathcal B_{\mathcal A})=\sum_{k=1}^\infty \lambda^{k-1}Q_{\mathcal A}\mathcal B(z_k)
$$
are absolutely convergent for $\vert \lambda\vert <\frac{\xi}{\sigma +\xi\tau}$. The convergence of the last series follows from the fact that $\Vert Q_{\mathcal A}\mathcal B(z_k)\Vert \leq \Vert Q_{\mathcal B}\mathcal B(z_k)\Vert$ since $\mathcal B(0)\subset \mathcal A(0)$.

Let $u_n(\lambda)$, $\lambda_n(\mathcal A)$, $\lambda_n(\mathcal B)$ and $\lambda_n(\mathcal B_{\mathcal A})$ denote the sequences of the partial sums of the above series in that order. Then for each $n$, $u_n(\lambda)\in D(\mathcal)$ and $\lambda_n(\mathcal A)\in \widetilde{Y}:=Y\slash \mathcal A(0)$. Furthermore, $u_n(\lambda)\to u(\lambda)$ and $\lambda_n(\mathcal A)\to \lambda(\mathcal A)$. Since $Q_{\mathcal A}\mathcal A$ is closed by Lemma \ref{Equivalent-Closed} we see that $u(\lambda)\in D(Q_{\mathcal A}\mathcal A)=D(\mathcal A)$ and that
\begin{equation}\label{Equal-1}
Q_{\mathcal A}\mathcal A(u(\lambda))=\lambda(\mathcal A)=\sum_{k=1}^\infty \lambda^kQ_{\mathcal A}\mathcal A(z_{k+1}).
\end{equation}
Since $\mathcal A(z_{k+1})\cap \mathcal B(z_k)\ne \emptyset$, a similar argument shows that
\begin{equation}\label{Equal-2}
Q_{\mathcal A}\mathcal B(u(\lambda))=\lambda(\mathcal B_{\mathcal A})=\sum_{k=1}^\infty \lambda^{k-1}Q_{\mathcal A}\mathcal B(z_k) =
\sum_{k=1}^\infty \lambda^kQ_{\mathcal A}\mathcal A(z_{k+1})=\lambda(\mathcal A).
\end{equation}
One also obtains the equality $Q_{\mathcal B}\mathcal B(u(\lambda))=\lambda(\mathcal B)=\sum_{k=1}^\infty \lambda^kQ_{\mathcal B}\mathcal B(z_k)$ using the closedness of $\mathcal B$.

From \eqref{Equal-1} and \eqref{Equal-2} we see that
$$
Q_{\mathcal A}[\mathcal A(u(\lambda))-\lambda\mathcal B(u(\lambda))] = \widetilde{0}
$$
and so $u(\lambda)\in N(\mathcal A-\lambda\mathcal B)$.

Furthermore,
$$
\Vert u(\lambda)-z_1\Vert \leq \sum_{k=1}^\infty \vert \lambda\vert^{k-1}\Vert z_k\Vert \leq
\left( \frac{\sigma\vert \lambda\vert}{\xi-\vert \lambda\vert(\sigma +\tau\xi)} \right)\Vert z_1\Vert.
$$
Since there is such a $u(\lambda)\in N(\mathcal A-\lambda\mathcal B)$ for every $z-z_1\in N(\mathcal A)$, we conclude that
\begin{equation}\label{Conclude}
\delta(N(\mathcal A), N(\mathcal A-\lambda\mathcal B))\leq \frac{\sigma\vert \lambda\vert}{\gamma(\mathcal A)-\vert \lambda\vert(\sigma +\tau\gamma(\mathcal A))}.
\end{equation}
\end{proof}

We observe that if $\alpha(\mathcal A)<\infty$ then Theorem \ref{Complete} can be used to conclude that $\mathcal A-\lambda \mathcal B$ has closed range if $\vert \lambda\vert < \frac{\gamma(\mathcal A)}{\sigma +\tau\gamma(\mathcal A)}$. However, this conclusion is not possible if no restriction is imposed on $\alpha(\mathcal A)$. This case is considered in the next lemma.

\begin{lemma}\label{Closed range-A-B}
Let $\mathcal A$ and $\mathcal B$ be as in {\rm Theorem \ref{Stabilty in Lambda-1}} with $\nu(\mathcal A:\mathcal B)=\infty$. Then $\mathcal A-\lambda\mathcal B$ has closed range for $\vert \lambda\vert < \frac{\gamma(\mathcal A)}{3\sigma+\tau\gamma(\mathcal A)}$.
\end{lemma}
\begin{proof}

In the present case, let $x\in X$ and set $y=x-u$ for any $u\in N(\mathcal A-\lambda\mathcal B)$.  Lemma \ref{Extra} implies that for any $\varepsilon>0$,
\begin{equation}\label{Imply}
\Vert \widetilde{y}\Vert ={\rm dist}\; (y,N(\mathcal A))\geq \frac{1-\delta(N(\mathcal A), N(\mathcal A-\lambda\mathcal B)}{1+\delta(N(\mathcal A), N(\mathcal A-\lambda\mathcal B}(1-\varepsilon)\Vert y\Vert.
\end{equation}
Suppose that $x\in D(\mathcal A)=D(\mathcal A-\lambda\mathcal B)$ and let $\delta:=\delta(N(\mathcal A):N(\mathcal A-\lambda\mathcal B))$. Since $(\mathcal A-\lambda\mathcal B)(u)=(\mathcal A-\lambda\mathcal B)(0)=\mathcal A(0)$, we see that
\begin{eqnarray}\label{Array-Norm}
\Vert (\mathcal A-\lambda\mathcal B)(x)\Vert &=& \Vert (\mathcal A-\lambda\mathcal B)(y)\Vert \geq \Vert \mathcal A(y)\Vert -\vert \lambda\vert\;\Vert \mathcal B(y)\Vert \; \; ({\rm by\; Lemma\; \ref{Norm Difference}}) \\
&\geq& \Vert \mathcal A(y)\Vert -\vert \lambda\vert(\sigma\Vert y\Vert+\tau\Vert \mathcal A(y)\Vert \nonumber\\
&=& (1-\tau\vert \lambda\vert) \Vert \mathcal A(y)\Vert - \sigma\vert \lambda\vert\; \Vert y\Vert \nonumber\\
&\geq & (1-\tau\vert \lambda\vert) \gamma(\mathcal A)\Vert \widetilde{y}\Vert - \sigma\vert \lambda\vert\; \Vert y\Vert \nonumber\\
&\geq& (1-\tau\vert \lambda\vert) \gamma(\mathcal A)\left(\frac{1-\delta}{1+\delta}\right)(1-\varepsilon)\Vert y\Vert  - \sigma\vert \lambda\vert\; \Vert y\Vert\; \; ({\rm by\; \eqref{Imply}}) \nonumber\\
&\geq& [\gamma(\mathcal A)-(2\sigma+\tau\gamma(\mathcal A))\vert\lambda\vert](1-\varepsilon)\Vert y\Vert-\sigma\vert \lambda\vert\; \Vert y\Vert\; \; ({\rm by\; \eqref{Finishing}}) \nonumber\\
&=&  [(\gamma(\mathcal A)-(2\sigma+\tau\gamma(\mathcal A))\vert\lambda\vert)(1-\varepsilon)-\sigma\vert \lambda\vert]\;\Vert y\Vert. \nonumber
\end{eqnarray}
Let $\widehat{X}$ denote the quotient space $X\slash N(\mathcal A-\lambda\mathcal B)$. Since $x-y=u\in N(\mathcal A-\lambda\mathcal B)$, we see that $\Vert y\Vert \geq \Vert \widehat{y}\Vert =\Vert \widehat{x}\Vert$ and therefore \eqref{Array-Norm} implies that
\begin{equation}\label{Go-to-zero}
\Vert (\mathcal A-\lambda\mathcal B)(x)\Vert \geq [(\gamma(\mathcal A)-(2\sigma+\tau\gamma(\mathcal A))\vert\lambda\vert)(1-\varepsilon)-\sigma\vert \lambda\vert ]\; \Vert \widehat{x}\Vert.
\end{equation}
Letting $\varepsilon \to 0$ in \eqref{Go-to-zero} leads to the inequality
\begin{equation}
\Vert (\mathcal A-\lambda\mathcal B)(x)\Vert \geq [(\gamma(\mathcal A)-(2\sigma+\tau\gamma(\mathcal A))\vert\lambda\vert)-\sigma\vert \lambda\vert ]\; \Vert \widehat{x}\Vert,
\end{equation}
from which we conclude that
$$
\gamma(\mathcal A-\lambda\mathcal B)\geq (\gamma(\mathcal A)-(3\sigma+\tau\gamma(\mathcal A))\vert\lambda\vert).
$$
It therefore follows that $\gamma(\mathcal A-\lambda\mathcal B)>0$ and therefore $R(\mathcal A-\lambda\mathcal B)$ is closed if
$\vert \lambda\vert < \frac{\gamma(\mathcal A)}{3\sigma+\tau\gamma(\mathcal A)}$.
\end{proof}
Finally, we establish the stability of both the nullity and deficiency of $\mathcal A-\lambda \mathcal B$ for $\lambda$ inside the disk $\vert \lambda\vert <\rho$ for some constant $\rho$.

\begin{theorem}\label{Stabilty in Lambda}
Let $X$ and $Y$ be Banach spaces and let $\mathcal A, \mathcal B\in CLR(X,Y)$ be such that $\mathcal A$ has closed range, $D(\mathcal B)\supset D(\mathcal A)$, $\mathcal B(0)\subset \mathcal T(0)$, and
\begin{equation}\label{Norm-Again-2}
\Vert \mathcal B(x)\Vert \leq \sigma \Vert x\Vert +\tau \Vert \mathcal A(x)\Vert \;\; {\rm for\; every}\; x\in D(\mathcal A),
\end{equation}
where $\sigma$ and $\tau$ are non negative constants. If $\nu(\mathcal A:\mathcal B)=\infty$ then $\alpha(\mathcal A-\lambda \mathcal B)$ and $\beta(\mathcal A-\lambda\mathcal B)$ are constants for all $\lambda$ for which $\vert \lambda \vert<\frac{\gamma(\mathcal A)}{3\sigma +\tau\gamma(\mathcal A)}$.
\end{theorem}
\begin{proof}
Let $u\in N(\mathcal A-\lambda\mathcal B)$. Then $\mathcal A(u)\cap \lambda\mathcal B(u)\ne \emptyset$ and we see from \eqref{Not Empty} that
$$
\Vert \widetilde{u}\Vert \leq \sigma\vert \lambda\vert \Vert u \Vert/(1-\vert \lambda\vert\tau)\gamma(\mathcal A).
$$
Since $\Vert \widetilde{u}\Vert = {\rm dist}\;(u, N(\mathcal A)$, we see from characterization \eqref{Characterize} that
$$
\delta(N(\mathcal A-\lambda\mathcal B), N(\mathcal A))\leq \frac{\sigma\vert \lambda\vert}{(1-\vert \lambda\vert\tau)\gamma(\mathcal A)}.
$$
Since $\frac{\sigma\vert \lambda\vert}{(1-\vert \lambda\vert\tau)\gamma(\mathcal A)}$ if
$\vert \lambda\vert < \frac{\gamma(\mathcal A)}{\sigma +\tau\gamma(\mathcal A)}$, Lemma \ref{Kato page 200} implies that
\begin{equation}\label{One direction}
\alpha(\mathcal A-\lambda\mathcal B) \leq \alpha(\mathcal A) \; \; {\rm for}\; \;  \vert \lambda\vert < \frac{\gamma(\mathcal A)}{\sigma +\tau\gamma(\mathcal A)}.
\end{equation}

The reverse inequality follows from Theorem \ref{Stabilty in Lambda-1} by noting that the righthand side of \eqref{Finishing} is less than one if $\vert \lambda\vert <\frac{\gamma(\mathcal A)}{2\sigma +\tau\gamma(\mathcal A)}$. We therefore conclude by Lemma \ref{Kato page 200} that $\alpha(\mathcal A)\leq \alpha(\mathcal A-\lambda\mathcal B)$ if $\vert \lambda\vert <\frac{\gamma(\mathcal A)}{2\sigma +\tau\gamma(\mathcal A)}$. Combined with \eqref{One direction} we conclude that
\begin{equation}\label{Conclusion}
\alpha(\mathcal A)=\alpha(\mathcal A-\lambda\mathcal B)\; {\rm for}\;  \vert \lambda\vert <\frac{\gamma(\mathcal A)}{2\sigma +\tau\gamma(\mathcal A)}.
\end{equation}

To show that $\beta(\mathcal A-\lambda\mathcal)=\beta(\mathcal A)$, we make use of the linear relations $\breve{\mathcal A}$ and $\breve{\mathcal B}$ as defined in the proof of Theorem \ref{Complete}. Since $\breve{\mathcal A}$ is bounded, Lemmas \ref{Norms}$(c)$, \ref{Domains} $(d)$, and \ref{Lamma-Closed Range} imply that $R(\breve{\mathcal A}^\prime)$ has closed range. Since $\breve{\mathcal B}(0)^\prime \subset \breve{\mathcal A}(0)^\prime$ by Remark \ref{Contain-Prime} and  $\nu(\breve{\mathcal A}^\prime: \breve{\mathcal B}^\prime)=\infty$ by Lemma \ref{Nu}, all the assumptions of Theorem \ref{Stabilty in Lambda} are satisfied by the pair $\breve{\mathcal A}^\prime$ and $\breve{\mathcal B}^\prime$. Since $\Vert {\breve{\mathcal B}}^\prime\Vert = \Vert {\breve{\mathcal B}}\Vert <1$ by Lemmas \ref{Norms} $(a)$ and \ref{Domains} $(c)$,  it follows from \eqref{Conclusion} that
\begin{equation}\label{Breve}
\alpha({\breve{\mathcal A}}^\prime - \lambda{\breve{\mathcal B}}^\prime)=\alpha({\breve{\mathcal A}}^\prime)\;
{\rm for}\; \vert \lambda\vert <\frac{\gamma({\breve{\mathcal A}}^\prime)}{2\Vert {\breve{\mathcal B}}^\prime\Vert}.
\end{equation}
Since
$(\breve{\mathcal A}-\lambda\breve{\mathcal B})^\prime= ({\breve{\mathcal A}}^\prime-\lambda{\breve{\mathcal B}}^\prime)$ by Lemma \ref{Norms} $(b)$ and $(d)$ and $(\breve{\mathcal A}-\lambda\breve{\mathcal B})$ has closed range (since $\mathcal A-\lambda \mathcal B $ has closed range), it follows from \eqref{Same}, Lemma \ref{Equality}  and \eqref{Breve} that
$$
\beta(\mathcal A-\lambda \mathcal B)=\beta(\breve{\mathcal A}-\lambda\breve{\mathcal B})=\alpha({\breve{\mathcal A}}^\prime - \lambda{\breve{\mathcal B}}^\prime)=\alpha({\breve{\mathcal A}}^\prime)=\beta(\breve{\mathcal B})=\beta(\mathcal B).
$$
\end{proof}
Theorem \ref{Stabilty in Lambda} remains true if we replace the requirement $\nu(\mathcal A:\mathcal B)=\infty$ with $\mathcal B^{-1}(0)\subset \mathcal A^{-1}(0)$.

\bibliographystyle{plain}
\bibliography{Wanjala_Luliro}
\end{document}